\documentclass[12pt]{article}

\makeatletter
\newenvironment{sqcases}{%
  \matrix@check\sqcases\env@sqcases
}{%
  \endarray\right.%
}
\def\env@sqcases{%
  \let\@ifnextchar\new@ifnextchar
  \left\lbrack
  \def\arraystretch{1.2}%
  \array{@{}l@{\quad}l@{}}%
}
\makeatother

\usepackage{amsmath}
\usepackage{algpseudocode}

\usepackage{amsthm}
\usepackage{amsfonts}
\usepackage{ amssymb }
\usepackage{authblk}

\usepackage{wrapfig}

\usepackage{mdframed}
\usepackage{hyperref}

\usepackage[english]{babel}
\usepackage{url}

\usepackage{bbm} %%Letters with stripe like \R

\usepackage{centernot} %\cetnernot command gives more centered cross-out 

\usepackage{listings}
\usepackage[most]{tcolorbox}

\usepackage{inconsolata}
\usepackage{esvect} %% vectors
\usepackage{enumitem} %%itemizing, enumeration

\usepackage{graphicx} %%enhanced support of graphics

\usepackage{mathtools} %%beter math typesetting
\usepackage{float}
\usepackage{physics} %% things like \sin and \abs

\usepackage[colorinlistoftodos]{todonotes} %% \textcolor{teal}{text} writes text in red
\usepackage{xcolor} %loads color in 
\usepackage[utf8]{inputenc}
\usepackage[T1]{fontenc}
\usepackage{amsmath}
\usepackage{amsfonts}
\usepackage{amssymb}
\usepackage{graphicx}
\usepackage{subcaption}

\usepackage[utf8]{inputenc}
\usepackage[T1]{fontenc}
\usepackage{CormorantGaramond} %%font

\theoremstyle{plain} %%make text wider 
\usepackage{geometry}
\author{\sc Nina Zubrilina \\  Massachusetts Institute of Technology \\182 Memorial Dr., Simons Building, $2$-$336$ \\ Cambridge, MA, USA\\email:  nina57@mit.edu}
\date{}

\makeatletter
\newcommand{\pushright}[1]{\ifmeasuring@#1\else\omit\hfill$\displaystyle#1$\fi\ignorespaces}
\newcommand{\pushleft}[1]{\ifmeasuring@#1\else\omit$\displaystyle#1$\hfill\fi\ignorespaces}
\makeatother

\DeclarePairedDelimiter\floor{\lfloor}{\rfloor}

\begin{document}

\definecolor{test}{RGB}{100,160,180} 

\def\nina#1{\textcolor{test}{\upshape {#1}\upshape}}
\def\red#1{\textcolor{red}{\upshape {#1}\upshape}}

\newtheorem{thm}{Theorem}
\newtheorem{cor}[thm]{Corollary}
\newtheorem{prp}{Proposition}
\newtheorem*{est}{Theorem}
\newtheorem*{sv}{Large Sieve Inequality}
\newtheorem{theorem}{Theorem}[section]
\newtheorem*{theoremnonum}{Theorem}
\newtheorem*{cornonum}{Corollary}
\newtheorem{observation}[theorem]{Observation}
\newtheorem{corollary}[theorem]{Corollary}
\newtheorem{corol*}{Corollary}
\newtheorem{lemma}[theorem]{Lemma}
\newtheorem{prop}[theorem]{Proposition}
\newtheorem{prob}[theorem]{Problem}
\newtheorem{lemma*}{Lemma}
\newtheorem{defn}[theorem]{Definition}
\newtheorem{de*}{Definition}
\newtheorem{remark}[theorem]{Remark}
\newtheorem*{rem}{Remark}
\newtheorem*{lem}{Lemma}

\newenvironment{letter}{\begin{enumerate}[a)]}{\end{enumerate}}

\newlist{num}{enumerate}{1}
\setlist[num, 1]{label = (\alph*)}
\newcommand{\myitem}{\item}

\newcommand{\Mod}[1]{\hspace{0.02in} \mathrm{mod}\hspace{0.02in}  #1}

\renewcommand{\S}{\mathcal{T}_r}
\newcommand{\Prob}{\mathbb{P}}
\newcommand{\E}{\mathbb{E}}
\newcommand{\Q}{\mathbb{Q}}
\newcommand{\R}{\mathbb{R}}
\newcommand{\N}{\mathbb{N}}
\newcommand{\Z}{\mathbb{Z}}
\newcommand{\F}{\mathcal{F}}
\renewcommand{\P}{\mathcal{P}}
\newcommand{\C}{\mathbb{C}}
\newcommand{\Cl}{\operatorname{Cl}}
\newcommand{\disc}{\operatorname{disc}}
\newcommand{\Gal}{\operatorname{Gal}}
\newcommand{\Img}{\operatorname{Im}}
\renewcommand{\sl}{\operatorname{SL}}
\newcommand{\eps}{\varepsilon}
\newcommand{\ind}{\mathbbm{1}}
\renewcommand{\phi}{\varphi}
\newcommand{\ttl}{\mathrm{Tot}}
\newcommand{\vol}{\mathrm{Vol}}
\newcommand\ang[1]{\left\langle#1\right\rangle}

\renewcommand{\L}{\mathcal{L}}
\newcommand{\sym}{\mathrm{Sym}}

\renewcommand{\d}{\partial}
\newcommand*\Lapl{\mathop{}\!\mathbin\bigtriangleup}
\newcommand{\mat}[4] {\left(\begin{array}{cccc}#1 & #2\\ #3 & #4 \end{array}\right)}
\newcommand{\mattwo}[2] {\left(\begin{array}{cc}#1\\ #2 \end{array}\right)}
\newcommand{\pphi}{\varphi}
\newcommand{\eqmod}[1]{\overset{\bmod #1}{\equiv}}
\renewcommand{\O}{\mathrm{O}}
\renewcommand{\subset}{\subseteq}
\newcommand{\fr}[1]{\mathfrak{#1}}
\newcommand{\p}{\mathfrak{p}}
\newcommand{\m}{\mathfrak{m}}
\newcommand{\supp}{\mathrm{supp}}
\newcommand{\limto}[1]{\xrightarrow[#1]{}}
\newcommand{\sqf}{ \ \square\text{ - free}}

\newcommand{\under}[2]{\mathrel{\mathop{#2}\limits_{#1}}}

\newcommand{\underwithbrace}[2]{  \makebox[0pt][l]{$\smash{\underbrace{\phantom{%
    \begin{matrix}#2\end{matrix}}}_{\text{$#1$}}}$}#2}

\title{\sc Murmurations}
\maketitle
\begin{abstract}
We establish the first case of the surprising correlation phenomenon observed in the recent works of He, Lee, Oliver, Pozdnyakov, and Sutherland between Fourier coefficients in families of modular forms and their root numbers. We give a complete description of the resulting correlation functions for holomorphic modular forms of any fixed weight $k$ and examine the asymptotic properties of these functions.
\end{abstract}

\section{Introduction}
In a recent paper, He, Lee, Oliver, and Pozdnyakov (\cite{murms}) discovered a remarkable oscillation pattern in the averages of Frobenius traces of elliptic curves of fixed rank and conductor in a bounded interval. This discovery stemmed from the use of machine learning and computational techniques and did not explain the mathematical source of this phenomenon, referred to as "murmurations" due to its visual similarity to bird flight patterns: \\
\vspace{-0.1in} 

\begin{figure}[H]
  \begin{subfigure}[b]{3.2 in}
    \centering
    \includegraphics[width=3.2in]{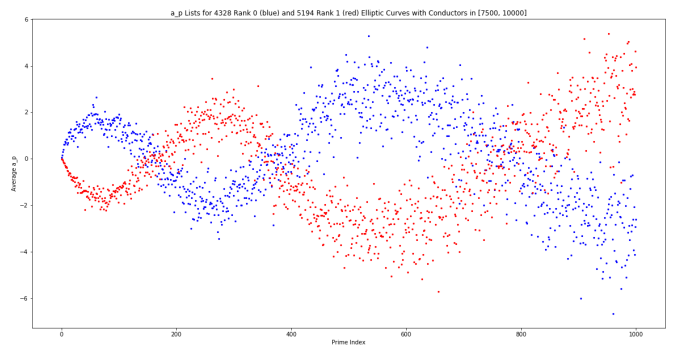}
 \caption{\footnotesize{Rank bias observed in \cite{murms}. With authors' permission.}}
  \end{subfigure}
  \hspace{0.11in}
  \begin{subfigure}[b]{3.3in}
    \centering
\vspace{-1in}
    \includegraphics[width=3.3 in]{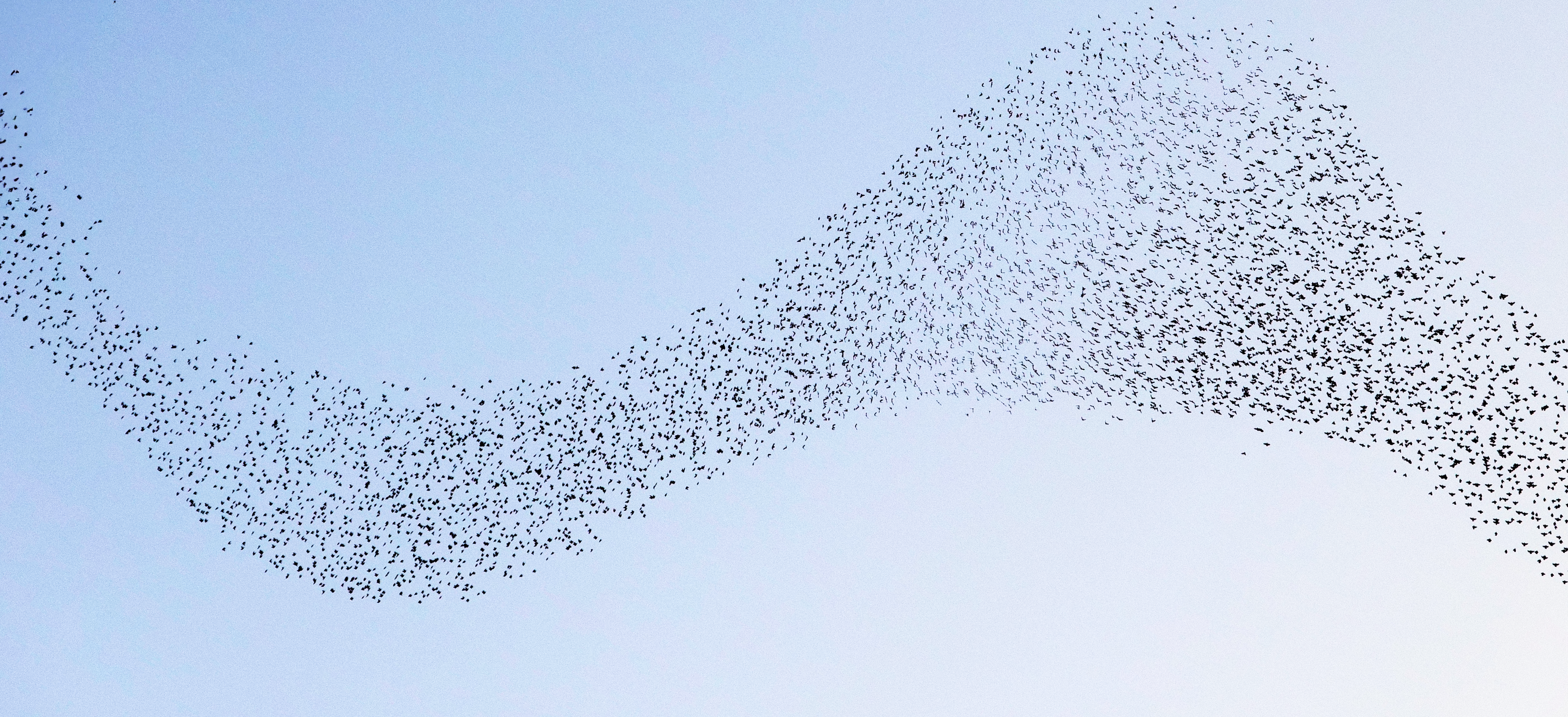}
  \caption{  \footnotesize{ Starlings murmurations.} \scriptsize{  AlbertoGonzalez/Shutterstock.com}}
  \end{subfigure}
\label{elcurves}
\end{figure}
\vspace{-0.1in}

Later, Sutherland and the authors (\cite{rubenst}, \cite{nextmurms}) detected this bias in more general families of arithmetic $L$ functions, for instance, those associated to weight $k$ holomorphic modular cusp forms for $\Gamma_0(N)$ with conductor in a geometric interval range $[N, cN]$ and a fixed root number. Sutherland made a striking observation that the average of $a_f(P)$ over this family for a single prime $P \sim N$ converges a continuous-looking function of $P/N$:\\

\begin{figure}[H]
  \centering
  % include first image
\includegraphics[width = \textwidth]{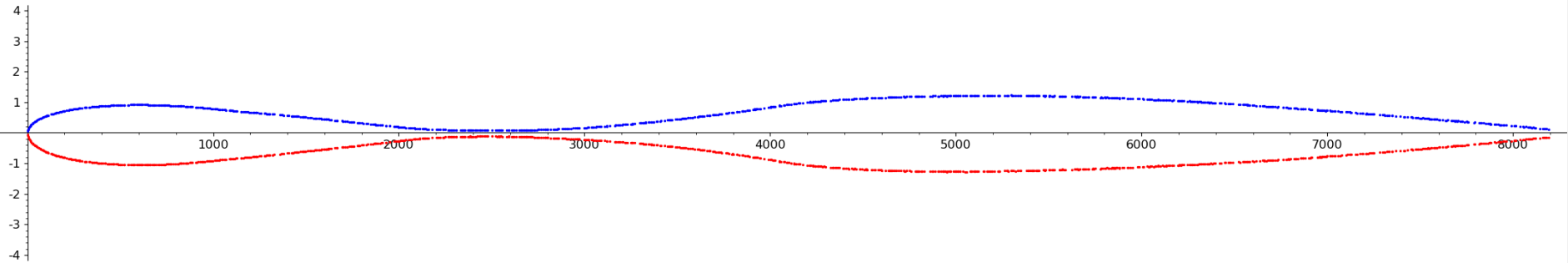}
  \caption{Averages of $a(p)$ for forms with a fixed root number over levels in $[2^{13}, 2^{14}]$, courtesy of Sutherland.}
\label{dr}
\end{figure}

The goal of this paper is to establish this bias in families of modular forms of square-free level with arbitrary fixed weight and root number. We show the following:

\begin{thm}\label{mainthm}
Let $H^{\text{new}}(N, k)$ be a Hecke basis for trivial character weight $k$ cusp newforms for $\Gamma_0(N)$ with $f \in H^{\text{new}}(N, k)$ normalized to have lead coefficient $1$. Let $\eps(f) \in \{\pm 1\}$ denote the root number of $f$, let $a_f(p)$ be the $p$-th Fourier coefficient of $f$, and let $\lambda_f(p) := a_f(p)/p^{(k-1)/2}$. Let $X, Y,$ and $P$ be parameters going to infinity with $X, Y \in \R_+$ and $P$ prime; assume further that $Y = (1 + o(1)) X^{1 - \delta_2}$, and $P \ll X^{1 + \delta_1}$ for some $\delta_1, \delta_2$ with $0 < \delta_1 < 1/11, \ \  2 \delta_1 < \delta_2 < 1/13 (4 - 18 \delta_1)$. Let $y := P/X$.  Then:

\begin{multline*}
 \frac{\sum^\square_{\substack{ N \in [X, X + Y] }}\sum_{f \in H^{\text{new}}(N, k)} \sqrt{P} \lambda_f(P)    \eps(f) }{ \sum^\square _{\substack{ N \in [X, X + Y]  }} \sum_{f \in H^{\text{new}}(N, k)} 1}   = \\ \\
  \frac{\alpha (-1)^{k/2 - 1}}{k - 1} \sum_{1 \leq r \leq 2 \sqrt{y}} \nu(r) \sqrt{4 y - r^2} U_{k - 2}\left(\frac{r}{2 \sqrt{y}} \right)  +  \frac{\beta}{k - 1} \sqrt{y}  - \gamma \delta_{k = 2}  y \\ + \O_\eps\left(  X^{-\delta' + \eps }  + \frac{1}{P}\right), 
\end{multline*}
where $\delta'>0$ is a constant exlpicitly expressible through $\delta_1, \delta_2$  \footnote{We remark that the exponents in the statement are far from optimal, as our goal here is only to get a power saving error term in a range up to $X^a$ for some $a > 1$.} , $U_{k - 2}$ is the Chebyshev polynomial given by 
\[U_n(\cos \theta) := \frac{\sin ((n + 1) \theta)}{\sin \theta},\] 
%\begin{center}
%\[A := \prod_p \left(1 + \frac{p}{(p+ 1)^2 (p-1)}\right), \ \  B:= \prod_p \frac{p^4 - 2 p^2 - p + 1}{(p^2-1)^2},\]
%\end{center} 
\begin{center}
\[\alpha = 2 \pi \prod_p \frac{1 - p - 2 p^2 + p^4}{p^4 - 2 p^2 + p}, \beta = 2 \pi  \prod_p \frac{-1 + p^2 + p^3}{p (-1 + p + p^2)}, \gamma =12 \prod_p \frac{p (1 + p)}{-1 + p + p^2},\] and
%\[\ \ \  \ D_k:= 2 \pi \frac{6  }{(k - 1) \pi^2 \prod_p\left(1 - \frac{1}{p^2 + p} \right)}, \ \ \ \
\[ \nu(r):=  \prod_{p | r}\left( 1 + \frac{p^2}{p^4 - 2 p^2 - p + 1}\right),\]
\end{center} and $\sum^\square$ denotes a sum over square-free parameters\footnote{The restriction to square-free levels is a technical one, as the trace formula simplifies greatly when the level is square-free. From the computations of Sutherland, it appears that the resulting density functions are slightly different when one considers all levels, but they share key properties with the ones above.}
\end{thm}   
We define
\[\mathcal{M}_k(y) :=  \frac{\alpha (-1)^{k/2 - 1}}{k - 1} \sum_{1 \leq r \leq 2 \sqrt{y}} \nu(r) \sqrt{4 y - r^2} U_{k - 2}\left(\frac{r}{2 \sqrt{y}} \right)  +  \frac{\beta}{k - 1} \sqrt{y}  - \gamma \delta_{k = 2}  y \]
to be the weight $k$ murmuration density.  

\newpage 

\begin{figure}[H]
  \centering
  % include first image
\includegraphics[width = \textwidth]{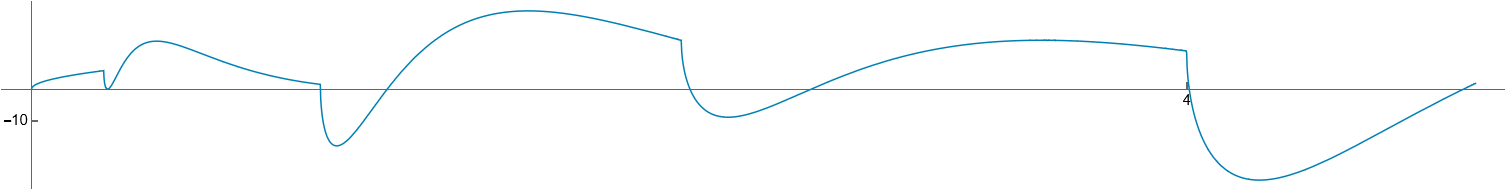}
\end{figure}

\begin{figure}[H]
  \centering
  % include first image
\includegraphics[width = \textwidth]{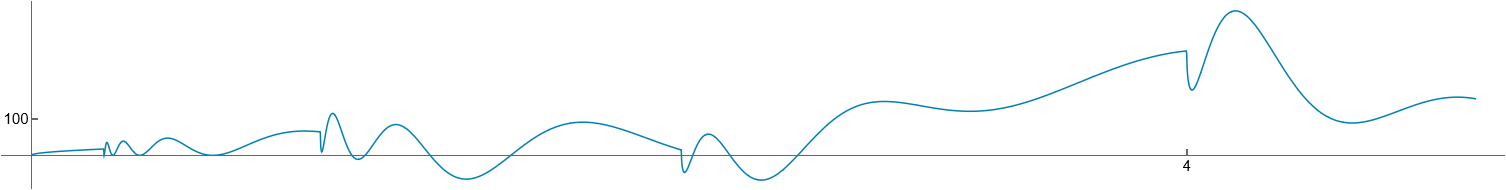}
  \caption{Plots of $\mathcal{M}_8$ and $\mathcal{M}_{24}$ near the origin.}
\label{m824} 
\end{figure}

The formula above arises from an application of the Eichler-Selberg trace formula to the composition of Hecke and Atkin-Lehner operators, which allows us to reinterpret the sum in terms of class numbers. We then compute class number averages in short intervals by means of the class number formula.
 
Integrating the murmuration density above produces the dyadic interval weight $k$ averages observed by Sutherland (Figure \ref*{dr}):

\begin{thm}\label{thm2}
Let $P \ll X^{6/5}$, let $c > 1$ be a constant, and let $y:= P/X$. Then as $X \to \infty,$

\[\frac{\sum^\square_{\substack{ N \in [X, cX] }}\sum_{f \in H^{\text{new}}(N, k)}  \sqrt{P}\lambda_f(P)   \eps(f)}{\sum^\square_{\substack{ N \in [X, cX]}}\sum_{f \in H^{\text{new}}(N, k)}  1}  = \frac{2}{(c^2 - 1)} \int_1^c u \mathcal{M}_k(y/u) d u + o_y(1),  \]
where $\mathcal{M}_k(y)$ is as in Theorem \ref*{mainthm}. 
In particular, for $k = c = 2$, the dyadic average
\[\frac{\sum^\square_{\substack{ N \in [X, 2X] }}\sum_{f \in H^{\text{new}}(N, 1)} a_f(P) \eps(f)}{\sum^\square_{\substack{ N \in [X, 2X]}}\sum_{f \in H^{\text{new}}(N, k)}  1}\]
converges to 
\[ \begin{cases} a \sqrt{y}  -b y & \text{on }[0, 1/4], \\   
a \sqrt{y}  - b y  + c \pi y^2  - c (1 - 2y) \sqrt{y - 1/4} - 2c y^2 \arcsin(1/2y - 1) ) & \text{on } [1/4, 1/2], \\ 
a \sqrt{y}  - b y  +   2c y^2 (\arcsin(1/y - 1) - \arcsin(1/2y - 1))\\
\hspace{0.76in} - c(1 - 2y) \sqrt{y - 1/4} + 2 c (1 - y) \sqrt{2y - 1} & \text{on } [1/2, 1], \end{cases} \]
where
$$a \approx 6.38936, b  \approx 11.3536, \text{ and }c \approx 2.6436$$ are explicit constants.
\end{thm}

\begin{figure}[H]
  \centering
  % include first image
\includegraphics[width = \textwidth]{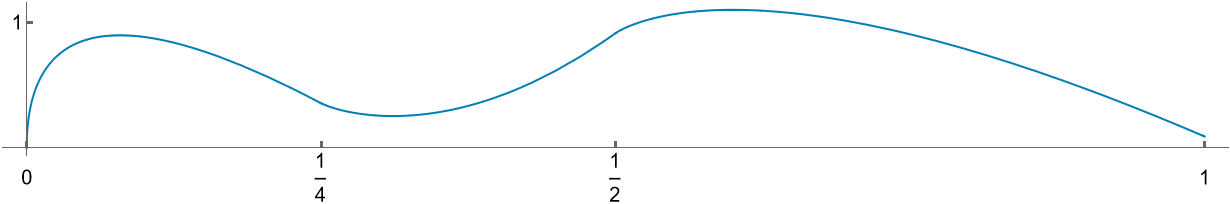}
  \caption{Plot of the piecewise defined function from Theorem \ref*{thm2}.}
\end{figure}

The murmuration densities $\mathcal{M}_k(y)$ in Theorem \ref*{mainthm} have many interesting features. They are oscillating, continuous, with derivative discontinuities at $n^2/4$ for $n \in \N$. At the origin, $\mathcal{M}_k(y)$s have a growth rate of $\sqrt{y}$ and are all positive. This positive root number bias for small $P$ has been observed previously in the works of Martin and Pharis (see \cite{kimball}, \cite{kimball1}, \cite{kimball2}). 

We analyze the behavior these functions as $y \to \infty$ in more detail. At infinity, it has a growth rate of $y^{1/4}$. The function $\mathcal{M}_k(y)/y^{1/4}$ is an asymptotically uniformly almost periodic function in $\sqrt{y}$; it is an absolutely convergent sum of periodic functions with growing (half-integer) periods. Furthermore, up to an $O(1)$ error term, all $\mathcal{M}_k(y)'s$ are given by the same function, with the sign changing depending on the pairity of $k/2$. All these features are captured by the following result:

\begin{thm}\label{tm3}
For any even $k > 0$, the weight $k$ murmuration density function is given by 
\[\mathcal{M}_k(y) =    \alpha \sqrt{y}  \sum_{d, s \in \N} Q(d)   \frac{J_{k - 1}(4 \pi s\sqrt{y}/d)}{s},\] where 
\[Q(d):= \mu^2(d) \prod_{p | d} \frac{p^2}{p^4 - 2 p^2 - p + 1} \asymp \mu^2(d)/d^2. \]  Asymptotically as $y \to \infty$,
\[\mathcal{M}_k(y) = y^{1/4}   \left((-1)^{k/2 - 1} \sqrt{2/\pi} \alpha \right) \sum_{d, s \in \N} \frac{Q(d) \sqrt{d}}{s^{3/2}} \cos \left(\frac{4 \pi s\sqrt{y}}{d} - \frac{3 \pi}{4} \right) + O(1).\] We note that the inner sum $\sum_{d, s \in \N} \frac{Q(d) \sqrt{d}}{s^{3/2}} \cos \left(\frac{4 \pi s\sqrt{y}}{d} - \frac{3 \pi}{4} \right)$ is absolutely convergent and uniformly bounded.
\end{thm}

\begin{figure}[H]
  \centering
  % include first image
\includegraphics[width =  \textwidth]{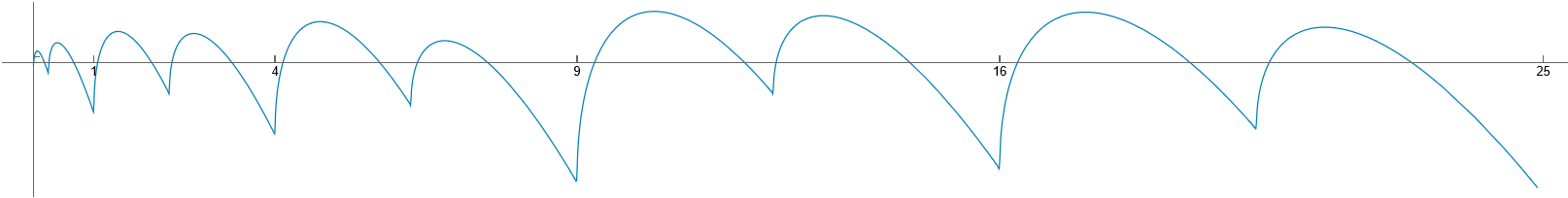}
  \caption{The universal limit function for $k = 2 \pmod 4$ from Theorem \ref*{tm3} near the origin.}
\includegraphics[width =  \textwidth]{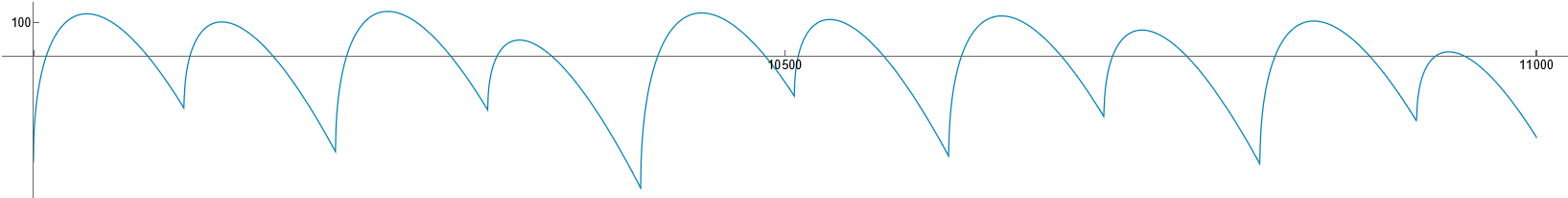}
  \caption{The graph of $-\mathcal{M}_8(y)$ in the range $y \in (10000, 11000)$ (see Figure \ref*{m824} for comparison).}
\end{figure}

From this reformulation, we can deduce computationally that the functions' sign, i.e., the sign of the correlation bias, changes infinitely often:

\begin{cor}\label{cor1}
For every even $k>0$, there exists a $y_0$ such that for all $y > y_0$, the function $\mathcal{M}_k(y^2)$ changes sign on every interval of length $1$.
\end{cor}

Finally, we analyze the asymptotic behavior of the smoothed averages from Theorem \ref*{thm2}:

\begin{thm}\label{thm4}
Let $\Phi:(0, \infty) \to \C$ be a compactly supported smooth weight function, and let 
\[\mathcal{M}^k_\Phi(y):=\Big( \int_0^\infty  \mathcal{M}_k(y/u) \Phi(u) u^2 \frac{du}{u}\Big)/\int_0^\infty   \Phi(u) u^2 \frac{du}{u}.\] Then $\mathcal{M}_\Phi$ is continuous on $(0, \infty)$, $\mathcal{M}^k_\Phi(0) = 0$, and as $y \to \infty$,
\[\mathcal{M}^k_\Phi (y) = \frac{1}{2}  + o_k(1).\]
\end{thm}
To address the asymptotic behavior of Figure \ref*{dr}, we also treat the case of smoothing by a characteristic function of an interval. For technical reasons, we need to assume RH and that $k \geq 6$ for our analysis, but these assumptions can likely be relaxed.
\begin{thm}\label{thm5}
Assume RH for $\zeta(s)$. Let $c > 1$ and $k \geq 6$, let
\[\mathcal{M}^k_c(y):=\left(\int_1^c  \mathcal{M}_k\left(\frac{y}{u}\right) u^2 \frac{du}{u}\right) /\left(\int_1^c   u^2 \frac{du}{u}\right).\] Then $\mathcal{M}_k^c$ is continuous on $(0, \infty)$, $\mathcal{M}_k^c(0) = 0$,  and as $y \to \infty$, \[\mathcal{M}^k_c(y) =   \frac{1}{2} + o_{ k}(1).\]
\end{thm}
Because the smoothed function tends to $1/2$, it can only change sign a finite number of times. Since the murmuration density function changes sign infinitely often, this implies that as values of $c$ varies between $1$ to infinity, the resulting smoothed murmuration function can have any prescribed finite number of zeros.

Murmurations are a feature of the the one-level density transition range. The Katz-Sarnak philosophy (\cite{KS}) predicts that averages as in Theorem \ref*{thm4} for $P \sim N^a$ behave differently when $a < 1$ and $a > 1$, and our statements for $P \sim N$ describe the phase transition between these ranges. The unusual (for $k > 2$) normalization of the coefficients above also arises naturally from this interpretation. For a more detailed discussion of this connection, we point the reader to \cite{peterletter}. 

Computing similar averages weighted "harmonically" (i.e., by the value at $1$ of the symmetric square $L$-function) by means of the Petersson formula reveals that with weights, this bias becomes much less pronounced: the resulting function grows like $y$, as opposed to $\sqrt{y}$, at the origin.

Murmurations for elliptic curves over $\Q$ are not explained by these results, as they constitute a very sparse subset of weight $2$ modular forms. We point out that the best fit curve to approximate the data of elliptic curve murmurations does not match the curve in Figure \ref*{dr}. We also point out that computational observations of the aforementioned authors make a compelling case that this phenomenon is very sensitive to the ordering by conductor, and disappears almost entirely when the curves are ordered by naive height, $j$-invariant, or discriminant.

\section{Trace Formula Setup}

Given a square-free positive integer $N$ and a prime $P \nmid N$, let $H^{\text{new}}(N, k)$ denote a Hecke eigenbasis of the space $S^{\text{new}}(N, k)$ of weight $k$ cusp newforms for $\Gamma_0(N)$. For $f \in H^{\text{new}}(N, k)$, let $f$ be normalized to have $a_f(1) = 1$, and let $a_f(P) = \lambda_f(P) P^{(k-1)/2}$ denote the eigenvalue of $f$ under the $P$-th Hecke operator. Let $\eps(f)$ denote the root number of $f$ (recall $(-1)^{k/2} \eps(f)$ is equal to the eigenvalue of $f$ under the Atkin-Lehner involution $W_N$). In order to compute the average of $ a_f(P) \eps(f)$ for eigenforms $f$ ranging over square-free levels $N$ in an interval, we interpret $\sum_{f \in H^{\text{new}}(N, k)} a_f(P) \eps(f)$ as the trace of the operator $(-1)^{k/2}T_p \circ W_N$ on $S^{\text{new}}(N, k)$ and apply the corresponding trace formula. 

Such a trace formula was first derived by Yamauchi in \cite{yamauchi}; the result contained a computational error which was later corrected by Skoruppa and Zagier (\cite{SkZ}). This formula (section $2$, formula ($7$)) gives the trace of $T_p \circ W_N$ on the full space of cusp forms $S(N, k)$. As the authors point out in the discussion leading to formula $(5)$, oldforms coming from $S(M, k)$ contribute to the trace only when $N/M$ is a square. Since we are restricting ourselves to $N$ square-free, we thus have the following result at our disposal:
\begin{theoremnonum}[Skoruppa-Zagier (\cite{SkZ}, section $2$, formulas ($5$) and ($7$)]
For $N$ square-free and a prime $P \nmid N$, 
\begin{align*}
\sum_{f \in H^{\text{new}}(N, k)} \sqrt{P} \lambda_f(P) \eps(f) &=    \frac{  H_1(-4 P N)}{2}+  (-1)^{k/2 - 1} U_{k - 2}\left(\frac{r\sqrt{N}}{2 \sqrt{P}} \right) \sum_{1 \leq r \leq 2 \sqrt{P/N}} H_1(r^2 N^2 - 4 P N)\\
&  -  \delta_{k = 2}(P + 1). \\
\end{align*}
\end{theoremnonum}
Here $U_k$ is the Chebyshev polynomial and $H$ is the Hurwitz class number, that is, $H_1(-d)$ is the number of equivalence classes with respect to $\sl_2(\Z)$ of positive-definite binary quadratic forms of discriminant $-d$ weighed by the number of automorphisms (i.e., with forms corresponding to multiples of $x^2 + y^2$ and $x^2 + xy + y^2$ counted with multiplicities $1/2$ and $1/3$, accordingly). $H_1$ can be expressed in terms of the Gauss class number $h$ via: 
$$H_1(-d) = \sum_{f \in \N: f^2 | d} h(-d/f^2) + \O(1),$$ with the error term disappearing if $d \neq 3\cdot \square, 4 \cdot \square$.

Assume from now on that $P > 2$ and $P \nmid N$. The square factors of $4 PN$ are $1$ and $4$, since by assumption $P \nmid N$. For a prime $q$ and $r \geq 1$, the condition $q^2 \mid  N(r^2N - 4P)$ can hold either if $q^2 | r^2N - 4 P$ or if $q$ divides both $N$ and $4P$, i.e., if $q = 2$ and $N$ is even. However, if $N = 2 \tilde{N}$ is even (with $\tilde{N}$ odd), then for any $d$ with $4d^2 | (r^2 N^2 - 4 PN)$, one has \[(r^2N^2 - 4 P N)/4d^2 = (r^2\tilde{N}^2 - 2 P \tilde{N})/d^2,\] which is always $2$ or $3$ modulo $4$, so the corresponding class number vanishes. Thus it suffices to consider square divisors of $r^2N^2 - 4 PN$ for which $d^2 | r^2 N - 4P$. In summary, for $N$ square-free, and a prime $P \nmid 2N$, the trace formula can be rewritten as

\begin{align}\label{traceformula}
\sum_{f \in H^{\text{new}}(N, k)} \sqrt{P} \lambda_f(P) \eps(f)  &=     \frac{ h(-4PN)}{2} + \frac{h(-PN)}{2}  - \delta_{k = 2} P + \O(1).  \nonumber\\
&+ (-1)^{k/2 -1 } U_{k - 2}\left(\frac{r\sqrt{N}}{2 \sqrt{P}} \right)\sum_{1 \leq r \leq 2 \sqrt{\frac{P}{N}}} \sum_{d^2 | r^2 N - 4 P} h (N (r^2N - 4P)/d^2).
\end{align}
 
From this formula, one can already see that the trace is positively biased when $P \leq N^{1 - \delta}$ for some $\delta >0$. Indeed, for $P/N$ near $0$, the only negative term in this expression is $-P$; on the other hand, Siegel's bound dictates that the class number terms should be of size $(PN)^{1/2 + \eps}$. This has been observed in (\cite{kimball}) and (\cite{kimball1}), (\cite{kimball2}). 

On the other hand, for $P$ of size $N^{1 +\eps }$, the balance becomes more subtle, and as we will see, the trace can be either positive or negative, even when averaged over short intervals in $N$.

\section{Average Class Number in Short Intervals}
Our interest in this section is to exploit the Dirichlet class number formula to understand sums of class numbers in (\ref*{traceformula}) as the square-free parameter $N$ ranges over a short interval $[X, X + Y]$ for $Y = o(X)$. For such an interval, the square root term in the class number formula has approximately fixed size, so these sums can be understood by averaging Dirichlet characters coming from a truncated $L$ function special value. Carrying out this computation yields Theorem \ref*{mainthm}. We establish it via the following two propositions: 
\begin{prop}\label{easyaverage}
Let $P > 2$ be prime and let $[X, X + Y]$ be an interval of length $Y = o(X)$. Let $y:= P/X$. Then as $X \to \infty$, 

\begin{multline}
\frac{\zeta(2) \pi}{X Y}  \sideset{}{^\square}\sum_{\substack{ N \in [X, X + Y] \\ P \nmid N }} \frac{h(-PN)}{2} + \frac{h(-4 P N)}{2}  =   A \sqrt{y}    + \O_\eps\Big(\frac{1}{P^{\frac{1}{2}}X^{\frac{1}{2}}} + \frac{P^{\frac{11}{19}}}{Y^{\frac{16}{19}}} +    \frac{ Y P^{\frac{1}{2}}}{X^{\frac{3}{2}}}  \Big) (XP)^\eps,
\end{multline}
where 
\[A := \prod_p \left(1 + \frac{p}{(p+ 1)^2 (p-1)}\right).\]
\end{prop}
\begin{prop}\label{hardaverage}

Let $P > 2$ be prime and let $[X, X + Y]$ be an interval of length $Y = o(X)$. Let $y:= P/X$. Then as $X \to \infty$, 
\begin{align*}
&\frac{\zeta(2) \pi}{Y X}\sum_{1 \leq r \leq 2 \sqrt{P/X}} \  \sideset{}{^\square}\sum_{\substack{ N \in [X, X + Y] \\ P \nmid N }} H_1(r^2 N^2 - 4 PN) =\sum_{1\leq r \leq 2 \sqrt{P/X}}B \nu(r)  \sqrt{4 y- r^2} \\
& + \O\Big(  \frac{ P^{\frac{11}{10}}}{Y^{\frac{2}{5}} X^{\frac{9}{10}}} +   \frac{Y P}{ X^2} + \frac{PY^{\frac{1}{2}}}{ X^{\frac{3}{2} }} +    \frac{P}{X^{\frac{1}{2}} Y^{\frac{13}{18}}} + \frac{P}{X Y^{\frac{1}{9}}} \Big) (XYP)^\eps,
\end{align*}
where 
\[ B:= \prod_p \frac{p^4 - 2 p^2 - p + 1}{(p^2-1)^2}\] and $\nu(r)$ are defined as in Theorem \ref*{mainthm}.
\end{prop}

Subsections \ref*{pn} proves Proposition \ref*{easyaverage}. In subsection \ref*{hellpn}, we prove Proposition \ref*{hardaverage} by adapting the same idea to the more complex square divisor structure of the arguments of the class numbers involved. Finally, in subsection \ref*{allotherpn} we collect all the results to prove Theorem \ref*{mainthm}.

\subsection{$H_1(-4PN)$}\label{pn}
We begin by stating some properties of multiplicative functions that will come up in evaluating the sum of $H_1(-4PN)$. After that, we compute sums over $h(-PN) $ and $h(-4PN)$ separately.
\subsection{Some Multiplicative Functions}
We let 

\begin{lemma}\label{phissums}
 Let $K$ be a cut-off parameter and let $P \neq 2$ be a prime. Let 
\[\eta(m):= \frac{m}{\psi(m)} = \prod_{p | m}\frac{p}{p + 1}.\]

\[\sum_{\substack{ m \text{ odd } \\ (P, m) =1 }}^K \frac{\eta(m)}{m^2} = \frac{9A}{11} + \O\left(\frac{1}{P^2} + \frac{1}{K^{1 - \eps}}\right); \ \ \sum_{\substack{ m :  (P, m) =1 }}^K  \hspace{-0.15in} \frac{\eta(2m) }{m^2}  =  \frac{8A}{11}+ \O\left(\frac{1}{P^2} + \frac{K^\eps}{K}\right).\]
\end{lemma}
\begin{proof}
Since $\eta(n) < n^\eps$ and since 
\[\sum_{\ell\geq 1} p_{-2\ell} \  \eta(\ell) = \frac{1}{p^2} \frac{1}{1 - p^{-2}} \frac{p}{(p + 1)}, \] we have
\[\sum_{\substack{ m \text{ odd } \\ (P, m) =1 }}^K \frac{\eta(m)}{m^2}  =\prod_{p \neq 2, P} \left(1 + \frac{p}{(p+1)^2(p-1)}\right) + \O\left( \frac{1}{K^{1 - \eps} }\right) =  \frac{9A}{11} + \O\left(\frac{1}{P^2} + \frac{1}{K^{1 - \eps}}\right). \]
For the second sum, writing $m = 2^\ell n$ for $n$ odd, we get
\[\sum_{\substack{ m :  (P, m) =1 }}^K  \hspace{-0.15in} \frac{\eta(2m) }{m^2}  = \frac{2}{3} \sum_{\ell \geq 0} 2^{-\ell}  \prod_{p \neq 2, P } \left(1 + \frac{p}{(p+1)^2(p-1)}\right) + O\left(\frac{K^\eps}{K}\right) =  \frac{8A}{11}+ \O\left(\frac{1}{P^2} + \frac{K^\eps}{K}\right). \]
\end{proof}

\begin{lemma}\label{twistedsum}
Let $m \in \N$, and let $\chi$ be the principal quadratic character modulo $m$. Then: 
$$ \sum_{\substack{ N\leq Z}}  \mu^2(N) \chi(N) =  \frac{Z}{\zeta(2)}\prod_{p | m}\frac{p}{p + 1}   + O_{\eps} \left(\sqrt{Z}m^\eps \right) = Z \cdot \frac{ \eta(m)}{\zeta(2)}   + O_{\eps} \left(\sqrt{Z} m^\eps \right)$$
\end{lemma}

\begin{proof}
We have
\begin{align*}
\sum_{N \leq Z} \mu^2(N) \chi(N) &= \sum_{d | m}\mu(d) \sum_{\substack{n \leq \sqrt{Z/d} \\ (n, m) = 1}} \mu(n) \left\lfloor\frac{Z}{n^2 d}\right\rfloor= \sum_{d | m}\mu(d) \sum_{\substack{n \leq \sqrt{Z/d} \\ (n, m) = 1}} \mu(n) \frac{Z}{n^2 d} + O\left( \sum_{d | m} \sqrt{Z/d}\right)\\
&=Z\sum_{d | m}\frac{\mu(d)}{d} \sum_{\substack{n \in \N \\ (n, m) = 1}} \frac{\mu(n)}{n^2}+ O\left(Z \sum_{d | m}\frac{1}{d} \left(\sqrt{\frac{d}{Z}} m^\eps \right) + \sqrt{Z} m^\eps\right)\\
&=Z\sum_{d | m}\frac{\mu(d)}{d} \sum_{\substack{n \in \N \\ (n, m) = 1}} \frac{\mu(n)}{n^2}+ O\left( \sqrt{Z} m^\eps\right).
\end{align*}
Now, 
\[\sum_{(m, n) = 1} \mu(n)/n^2 = \prod_{p \nmid m}\left(1 - \frac{1}{p^2} \right),\] and 
\[\sum_{d |m} \frac{\mu(d)}{d} = \prod_{p | m}\left( 1- \frac{1}{p}\right),\]
so their product simplifies to 
\[\frac{1}{\zeta(2)} \prod_{p |m}\frac{p-1}{p}\frac{p^2}{p^2 - 1} = \frac{\eta(m)}{\zeta(2)},\] as aimed. 
\end{proof}
We will also use a result of Burgess for an analogous bound for non-principal $\chi$:
\begin{theorem}[Burgess, \cite{Burg}]\label{burg}

Let $m \in \N$, and let $\chi$ be a non-principal quadratic character modulo $m$. Then: 
$$  \sum_{\substack{ N\leq Z}}  \mu^2(N) \chi(N) = O_{\eps} \left(\sqrt{Z} m^{3/16 + \eps} \right).$$
\end{theorem}

\subsubsection{Averages of $h(-PN)$}
By Dirichlet's class number formula, for $d > 4$, $h(-d)$ is $0$ if $-d = 2$ or $3 \Mod{4}$, and otherwise 
$$h(-d) = \frac{\sqrt{d}}{\pi} L(1, \chi_d), $$ where $L(1, \chi_d)$ is the value at $1$ of the Dirichlet series for the Kronecker symbol $\left( \frac{d}{n}\right)$ (a quadratic Dirichlet character of modulus $d$ or $4d$). We evaluate the sum
\begin{align}\label{twostars}
\frac{1}{\sqrt{PX}} \  \sideset{}{^\square}\sum_{\substack{ N \in [X, X + Y] \\ P \nmid N }} h(-PN) =\frac{1}{\pi} \  \sideset{}{^\square}\sum_{\substack{  N \in [X, X + Y] \\ PN =3 \Mod{4}\\ P \nmid N }} \sqrt{N/X} L(1, \chi_{-PN})
\end{align} 
 by truncating the Dirichlet series of the $L$-function and splitting the appearing Legendre symbols into principal and non-principal ones. 
For $\chi$ a non-principal Dirichlet character of modulus $d$, it follows from Abel summation and Polya-Vinogradov that for a truncation parameter $T$,
\begin{align}\label{truncate}
L(1, \chi) =\sum_{n \geq 1}\frac{ \chi(n)}{n} =\sum_{n = 1}^T \frac{ \chi(n)}{n} + \O\left( \sqrt{d} \log d/T \right).
\end{align}
Since $\chi_{-PN}$ is always a non-principal Dirichlet character for square-free $N$ with $P \nmid N$, we have
\begin{align*}
& \ &&\  \sideset{}{^\square}\sum_{\substack{  N \in [X, X + Y] \\ PN =3 \Mod{4}\\ P \nmid N }} \sqrt{N/X} L(1, \chi_{- P N}) 
= \  \sideset{}{^\square}\sum_{\substack{  N \in [X, X + Y] \\ PN =3 \Mod{4}\\ P \nmid N }}\sqrt{N/X}\sum_{n =1}^T \frac{  \left( \frac{-PN}{n}\right)}{n} + \O\left(\frac{Y  (PX)^{\frac{1}{2} + \eps}}{T}\right) \nonumber \\
&=&&\  \sideset{}{^\square}\sum_{\substack{  N \in [X, X + Y] \\ PN =3 \Mod{4}\\ P \nmid N }} \sum_{m =1}^{\sqrt{T}} \frac{ \sqrt{N/X} \left( \frac{-PN}{m^2}\right)}{m^2} + \  \sideset{}{^\square}\sum_{\substack{  N \in [X, X + Y] \\ PN =3 \Mod{4}\\ P \nmid N }} \sum_{\substack{ n =1\\ n \neq \square }}^T \frac{  \sqrt{N/X} \left( \frac{-PN}{n}\right)}{n}   +\O\left(\frac{Y (PX)^{\frac{1}{2} + \eps} }{T}\right) \nonumber \\
&=: &&  \mathrm{Sq} + \mathrm{NSq} +\O\left(\frac{Y (PX)^{\frac{1}{2} + \eps}}{T}\right) .
\end{align*}
Now,
\begin{align*}
\mathrm{Sq} &=  \sum_{m =1}^{\sqrt{T}} \frac{1}{m^2}\sum_{\substack{  N \in [X, X + Y] \\ PN =3 \Mod{4}\\ P \nmid N }}\mu^2(N) \left( \frac{-PN}{m^2}\right) \Big( 1 + \Big(\sqrt{1 + \frac{N - X}{X}} - 1\Big)\Big)\\
&=   \sum_{m =1}^{\sqrt{T}} \frac{1}{m^2}\sum_{\substack{  N \in [X, X + Y] \\ PN =3 \Mod{4}\\ P \nmid N }}\mu^2(N) \left( \frac{-PN}{m^2}\right)  + O\left(Y  \left(\sqrt{1 + Y/X} - 1\right) \right) \\
 &= \sum_{\substack{  m \leq \sqrt{T} \\ (P, m) =1  }}\frac{1}{m^2}  \left( \sum_{\substack{  N \in [X, X + Y]}} \mu^2(N)  \left( \frac{N}{m^2}\right) \frac{\chi_1(PN) - \chi_2(PN) }{2}\right)  + O\left(\frac{Y^2}{X} + \frac{Y}{P} + 1\right),
\end{align*}
where $\chi_{1, 2}$ are the characters modulo $4$, $\chi_1$ principal. The character $ \left( \frac{N}{m^2}\right) \chi_1(N)$ is principal modulo $2m$, and $  \left( \frac{N}{m^2}\right) \chi_2(N)$ is always non-principal modulo $4m$. Applying Lemmas \ref*{twistedsum} and Lemma \ref*{phissums},
\begin{align}\label{qmark}
\mathrm{Sq} &=  \sum_{\substack{  m \leq \sqrt{T} \\ (P, m) =1  }}\frac{Y }{\zeta(2)}\frac{\eta(2m)}{2m^2} + O_{\eps} \left(\sum_{m \leq \sqrt{T}} \frac{1}{m^2} m^{\frac{1}{5} + \eps} X^{\frac{3}{5} + \eps} \right) + O\left(\frac{Y^2}{X} +  \frac{Y}{P} + 1 \right) \nonumber \\
& = Y  \frac{4A }{11\zeta(2)} +    \O_\eps\left(\frac{Y}{P} + \frac{Y}{\sqrt{T}} +  X^{\frac{3}{5} + \eps} + \frac{Y^2}{X} \right).
\end{align}
Next, we bound the term 
\begin{align*}
\mathrm{NSq} &= \hspace{-0.2in}  \  \sideset{}{^\square}\sum_{\substack{  N \in [X, X + Y] \\ PN =3 \Mod{4}\\ P \nmid N }} \sum_{\substack{ n =1\\ n \neq \square }}^T \frac{\sqrt{N/X} \left( \frac{-PN}{n}\right)}{n}  =  \hspace{-0.2in} \  \sideset{}{^\square}\sum_{\substack{  N \in [X, X + Y] \\ PN =3 \Mod{4} }} \sum_{\substack{ n =1\\ n \neq \square }}^T \frac{\left( \frac{-PN}{n}\right)}{n} + O\left(\left(\frac{Y^2}{X} + \frac{Y}{P} + 1\right) \sum_{n \leq T} (1/n)\right)\\
& =  \sum_{\substack{ n =1\\ n \neq \square }}^T  \frac{\left( \frac{-P}{n}\right)}{n} \left( \  \sideset{}{^\square}   \sum_{\substack{  N \in [X, X + Y]  }}  \left( \frac{N}{n}\right) \frac{\chi_1(PN) - \chi_2(PN)}{2} \right) +  O\left(\frac{ Y^2 T^\eps}{X} + \frac{Y T^\eps}{P} + T^\eps\right)
\end{align*}
For $n$ not a square, $\left( \frac{N}{n}\right)$ is non-principal. Moreover, $\left( \frac{N}{2}\right)$ is primitive modulo $8$, so if the $2$-part of $\left( \frac{N}{n}\right)$ is non-principal, then so is the $2$-part of $\left( \frac{N}{n}\right) \chi_{1, 2}(N)$. Hence $\left( \frac{N}{n}\right) \chi_{1, 2}(N)$ are also non-principal, so applying Theorem \ref*{twistedsum},
\begin{align}\label{2qmark}
\mathrm{NSq} \ll_\eps X^{\frac{1}{2} + \eps} \sum_{\substack{ n =1\\ n \neq \square }}^T (1/n) n^{\frac{3}{16} + \eps}+ \frac{ Y^2 T^\eps}{X} + \frac{Y T^\eps}{P} + T^\eps  \ll_\eps (TX)^\eps \left(T^{\frac{3}{16} } X^{\frac{1}{2} } +\frac{ Y^2}{X} + \frac{Y}{P} \right).
\end{align}
Combining (\ref*{twostars}),  (\ref*{qmark}), and  (\ref*{2qmark}),
\begin{align}\label{tilda}
\frac{1}{\sqrt{PX}} \sum_{N \in [X, X + Y] }h(-PN) =  \frac{4 A}{11 \zeta(2)\pi} Y +  \mathrm{Err}_{Y, X, P, T},
\end{align} 
where 

$$\mathrm{Err}_{Y, X, P, T} \ll_\eps (PTX)^\eps\left(\frac{Y}{P} + \frac{Y}{\sqrt{T}} +  X^{\frac{1}{2}}T^{\frac{3}{16}} +\frac{ Y^2}{X} + \frac{Y(PX)^{\frac{1}{2}} }{T}\right).$$
In particular, setting $T := Y^{\frac{5}{6}} P^{\frac{5}{12}} X^{-\frac{1}{12}}$ and renormalizing (scaling by $(1/Y)\sqrt{P/X} $), we get an error term matching that of Proposition \ref*{easyaverage} (with the third term equal to the last one, and the second term disappearing as it is smaller than the $3$rd for such choice of $T$).
\subsubsection{Averages of $h(-4 PN)$}
We handle this case the same way as in the previous section. Since $-4PN$ is always $0 \Mod 4$, 
\begin{align*}
\frac{1}{\sqrt{PX}} \  \sideset{}{^\square}\sum_{\substack{  N \in [X, X + Y] \\ P \nmid N }} h(-4PN) &=\frac{2}{\pi}\  \sideset{}{^\square}\sum_{\substack{  N \in [X, X + Y] \\ P \nmid N }} \sqrt{N/X} L(1, \chi_{-4PN})\\
&=  \frac{2}{\pi} \sum_{\substack{  N \in [X, X + Y]  }}\sum_{n =1}^T \frac{  \left( \frac{-4PN}{n}\right)}{n}  + \O\left(\frac{Y (PX)^{\frac{1}{2} + \eps} }{T} + \frac{Y^2 T^\eps}{X} + \frac{Y}{P} + 1 \right).\\
\end{align*} 
Again, we can separate into principal and non-principal characters: 
\begin{align*}
\sum_{\substack{  N \in [X, X + Y]  }} \sum_{n =1}^T \frac{  \left( \frac{-4PN}{n}\right)}{n}  = \sum_{\substack{ n =1\\ n \neq \square \\ n \text{ odd } }}^T   \sum_{\substack{  N \in [X, X + Y]  }} \frac{ \left( \frac{-PN}{n}\right)}{n} +  \sum_{\substack{ m =1\\m \text{ odd } }}^{\sqrt{T}}\sum_{N \in [X, X + Y]} \frac{ \left( \frac{-PN}{m^2}\right)}{m^2}.
\end{align*}
Applying Lemma \ref*{twistedsum} and Lemma \ref*{phissums} as in the previous section, we conclude  
$$\frac{1}{\sqrt{PX}}\sum_{N \in [X, X + Y]} h(-4PN) = \frac{2 Y}{\pi\zeta(2)}  \sum_{\substack{m \text{ odd }\\ (m, P) = 1}}^{\sqrt{T}} \frac{\eta(m)}{m^2} + \mathrm{Err}_{Y, X, P, T} =Y  \frac{18A}{11\zeta(2) \pi} + \mathrm{Err}_{Y, X, P, T},$$ which finishes the proof of Proposition \ref*{easyaverage} in combination with (\ref*{tilda}).

\subsection{$H_1(r^2 N^2 - 4 P N$)}\label{hellpn}
The aim of this section is to prove Proposition \ref*{hardaverage}. We do that by establishing the following: 
\begin{prop}\label{hardaveragemid}
Let $P > 2$ be prime and let $[X, X + Y]$ be an interval of length $Y = o(X)$. Assume further that $r^2 ( X + Y) < 4 P.$
\footnote{The contribution of $r$'s with $r^2 ( X + Y) > 4 P$ will be bounded in subsection \ref*{propproof}.  } Then:
\begin{multline*}
 \sideset{}{^\square}\sum_{\substack{ N \in [X, X + Y] \\ P \nmid N }} H_1(r^2 N^2 - 4 PN) = \frac{Y B \nu(r)}{\pi\zeta(2)}  \sqrt{4 P X - r^2 X^2} \\
 + \O\big( (PX)^\eps \big( (YPX)^{\frac{3}{5}} +   Y^2 P^{\frac{1}{2}} X^{-\frac{1}{2}} + r Y^{\frac{3}{2}} X^{\frac{1}{2} } +    XP^{\frac{1}{2}} Y^{\frac{5}{18}} + Y^{\frac{8}{9}} (PX)^{\frac{1}{2}} \big)\big)
\end{multline*}
where as before, \[ B:= \prod_p \frac{p^4 - 2 p^2 - p + 1}{(p^2-1)^2}\] and $\nu(r)$ are defined as in Theorem \ref*{mainthm}.
\end{prop}

We prove this in subsections \ref*{rem} - \ref*{propproof}.

Subsection \ref*{rem} reduces the sum of Hurwitz class numbers to Gauss class numbers. In subsection \ref*{smf2}, we establish properties of multiplicative functions that arise in evaluating the sum\footnote{he computations of that section, especially Lemma \ref*{worstcomputation} involve combinatorial case work and routine Euler product summations, and are not particularly enlightning.} In subsection \ref*{truncation} we truncate the sum on $n$, reducing the question evaluating an expression of the form 
\[ \sum_{n \leq T, d }  \sum_{N}\mu^2(N)\left(\frac{N}{n}\right)  \left(\frac{(r^2N-4P)/d^2}{n}\right)/nd\]
where the sum on $N$ is restricted to certain congruence classes.

Since $\sum_{a \Mod m} \chi(Q(a))$ for a quadratic polynomial $Q$ and a non-principal $\chi$ modulo $m$ is not necessarily zero, we cannot split sums into principal and non-principal characters as in the previous section to evaluate this. Instead, our strategy is as follows:
\begin{itemize}
\item For $n \ll Y^\sigma$ for some $0 < \sigma < 1$, we compute the inner sum explicitly by exploiting equidistribution of square-free numbers in remainder classes modulo $d^2n$. This is Proposition \ref*{smalldn} proved in subsection \ref*{smalln}.
\item for $n \gg Y^\sigma$,  we upper bound the sum using Poisson summation.  This is Proposition \ref*{smalldlargem} proved in subsection \ref*{largen}.
\end{itemize}
In subsection \ref*{propproof}, we deduce Proposition \ref*{hardaveragemid}. Finally, in subsection \ref*{propdeduction}, we deduce Proposition \ref*{hardaverage} from Proposition \ref*{hardaveragemid}.

Throughout the following subsections, we assume $X, Y, P, T,$ and $r$ satisfy the assumptions of Proposition \ref*{hardaveragemid}.

\subsubsection{Remainder analysis}\label{rem}
For a divisor $d^2 | r^2 N - 4 P $ such that $\frac{r^2 N^2 - 4 P N}{d^2} = 0/1 \Mod 4$, we have by the class number formula that 
\[h\left(\frac{r^2N^2 - 4PN}{d^2}\right) = \frac{\sqrt{4 P N - r^2N^2}}{\pi d} L(1, \chi_{\frac{r^2N^2 - 4PN}{d^2}}).\] 
Thus, for $1 \leq r \leq2 \sqrt{P/(X + Y)}$,
\begin{align*} \sideset{}{^\square}\sum_{\substack{ N \in [X, X + Y] \\ P \nmid N }} H_1(r^2 N^2 - 4 PN) =  \sum_{\substack{ d^2 \leq 4 P  }}\sum_{\substack{ N \in [X, X + Y]  \\ N \in \widetilde{\mathcal{A}}_{r, d} }} \frac{L(1, \chi_{(r^2N^2 - 4PN)/d^2})}{\pi d}  \sqrt{4 P N - r^2N^2}, 
\end{align*}
where we define
\[\widetilde{\mathcal{A}}_{k,d} := \{ N \in \Z : N \sqf, P \nmid N, d^2 | r^2 N - 4P, \text{ and }(r^2 N^2 - 4 P N)/d^2 \equiv 0/1 \Mod 4 \}\] for $d^2 \leq 4 P$ and $1 \leq r \leq2 \sqrt{P/(X + Y)}$.  
In this section we analyze the set $\widetilde{\mathcal{A}}_{r, d}$.

Suppose first that $r$ is odd. For $N$ square-free, $d^2 | r^2N - 4 P$ implies that $d$ is odd, and that $\frac{r^2N^2 - 4PN}{d^2} \equiv 0/1 \Mod 4$ always holds. Thus, $\widetilde{\mathcal{A}}_{r, d}$ is the set of square-free integer solutions to the congruence \[r^2 N \equiv 4 P \Mod d^2, P \nmid N.\] For odd $r$ and $d$, this has a solution if and only if $P \nmid d$ and $(r, d) = 1$. Note that $d^2 \leq 4P$, so $(d, P) = 1$, and hence for such $d$,

\[ \widetilde{\mathcal{A}}_{r, d}= \{N \in  \mathcal{A}_{r, d}: N \sqf, P \nmid N\},\] where we let
\[ \mathcal{A}_{r, d}:= \begin{cases} n \in \Z : n \equiv 4 P r^{-2} \Mod d^2 &\text{ if } (d, r)  = (d, 2)  = 1; \\ \emptyset & \text{ otherwise} \end{cases}\] 
for odd $r$.

Now assume $r$ is even. From $d^2 |4P -  r^2 N < 4 P$ and since $r^2 X < 4 P$, we know $r, d < P$, i.e., $(P, r) = (P, d) = 1.$ Let $r:= 2 l$. Then $\frac{r^2 N^2 - 4 P N}{d^2}\equiv 0/1 \Mod 4$ is equivalent to the existance of some $t \in \Z$, $t N \equiv 0/1 \Mod 4$, and
\begin{align} \label{eqn}
4l^2 N \equiv 4P + t d^2 \Mod 4 d^2, P \nmid N.
\end{align}
If $d$ is odd, reducing (\ref*{eqn}) modulo $4$ shows that $t \equiv 0 \Mod 4$, and (\ref*{eqn}) is equivalent to \[l^2 N \equiv P \Mod d^2, P \nmid N.\] 
This has $1$ solution $\Mod d^2$ for $(l, d) = 1$ and no solutions otherwise.
Suppose now $d = 2b$, so (\ref*{eqn}) becomes 
\begin{align}\label{equat}
l^2 N \equiv P + t b^2 \Mod 4 b^2, tN \equiv 0/1\Mod 4, P \nmid N.
\end{align}
Since we restrict to $N$ square-free, we can disregard the case $N \equiv 0 \Mod 4$. If $N = 2 \Mod 4$, then the second condition forces that $t$ is even and (\ref*{equat}) has no solutions $\Mod 2$. For $N$ odd, $t N = 0/1 \Mod 4$ holds if and only if $t = 0/N \Mod 4$, and we have an equivalence
$$(\ref*{equat}) \ \& \  N   \text{ is odd } \iff \begin{sqcases} N(l^2 - b^2) \equiv P \Mod 4b^2, N \text { odd}, P \nmid N  \\ N l^2 \equiv P \Mod 4b^2, N \text{ odd}, P \nmid N. \end{sqcases}$$
If $(l, b) > 1$, this has no solutions since as we have previously noted, $P \nmid b$. Otherwise, if $(l, b) = 1$, there are three cases:

\begin{itemize}
\item if $l, b$ are odd, there is a solution $N \equiv P l^{-2} \Mod 4 b^2;$
\item if $l$ is even, $b$ is odd, there is a solution $N \equiv P (l^2 - b^2)^{-1} \Mod 4 b^2;$
\item if $l$ is odd, $b$ is even, there's a solution $N \equiv P l^{-2} \Mod 4 b^2$ and a solution $N \equiv P(l^2 - b^2)^{-1} \Mod 4 b^2$, distinct. 
\end{itemize}
Note all these congruences alone imply $N$ is odd.
In summary, for any choice of $r$ and $d$, 
\[\widetilde{\mathcal{A}}_{r, d}:= \{ N \in \mathcal{A}_{r, d}: P \nmid N, N \sqf\}, \] where the set $\mathcal{A}_{r, d}$ is given by a congruence condition modulo $d^2$; namely,
\[\mathcal{A}_{r, d}:= \{ n \in \Z : n \Mod d^2 \in \mathcal{R}_{d, r},\}\] where $\mathcal{R}_{d, r}$ is a subset of remainders mod $d^2$ coprime to $d$ that satisfies:

\[\abs{\mathcal{R}_{d, r}} = \begin{cases} 1 &\text{ if } (d, r) = 1, 2 \nmid rd; \\ 
 1 &\text{ if } (d, r) = 1, 2 |r; \\ 
1 &\text{ if } (d, r) = 2, 2 || d; \\
 2 &\text{ if } (d, r) = 2, 4 | d; \\ \emptyset & \text{ otherwise}.  \end{cases}\]

Furthermore, letting $s : = \frac{r^2 N^2 - 4 PN}{d^2}$ for some $N \in \widetilde{\mathcal{A}}_{r, d}$, the above analysis also implies that:
\begin{itemize}\label{parityofs}
\item For $(d, r) = 2, 2 || d, 2 ||r,$ $s$ is always even;
\item For $(d, r) = 2, 2 || d, 4 |r,$  $s$ is always odd; 
\item For $(d, r) = 2, 4 | d$, the two remainders in $\mathcal{R}_{r, d}$ produce $s$ of different parity.
\end{itemize}

We will call a pair $(r, d)$ \emph{admissible} if $\mathcal{R}_{r, d}$ is non-empty.

\subsubsection{Some Multiplicative Functions, II}\label{smf2}
\begin{defn}\label{multfuncdefn}
Let $\phi(m)$ be Euler's function and ${\psi}(m)$ be the Dedekind Psi function. As previously, we let $$\eta(m):= \frac{m}{\psi(m)} = \prod_{p | m}\frac{p}{p + 1}.$$ For any $r, m \in \N$ with $\upsilon_2(m) \neq 1, 2$ and a prime $P$, let 
\[ \theta_r(m) := \sum_{a \Mod m} \left(\frac{a}{m} \right) \left(\frac{ar^2 - 4P}{m} \right).\]
For an admissible pair $(r, d)$ and $g | d^\infty$, let
\[\phi^o_{r, d}(g) := \sum_{\substack{ a \Mod d^2 g \\ a  \Mod d^2 \in \mathcal{R}_{r, d}}}  \left(\frac{a}{g}\right) \left(\frac{(r^2a-4P)/d^2}{g}\right)\] where $\mathcal{R}_{r, d}$ is the set of remainders $\Mod d^2$ defined in the previous section, $d,g, r \in \N$, $\upsilon_2(g) \neq 1, 2$.
\end{defn}

\begin{lemma}\label{thetalemma}
Let $m \in \N$ with $\upsilon_2(m) \neq 1, 2$ and a prime $P \neq 2$ with $(m, P) = 1$, $ \theta_r(m)$ is a multiplicative function of $m$.
For odd $r$, it is given as follows. For a prime $p$ with $(p, 2r) = 1$,

$$\theta_r(p^\alpha)  = - p^{\alpha - 1} \text{ for }\alpha \text{ odd};\ \  \theta_r(p^\alpha)  = p^{\alpha - 1}(p-2) \text{ for } \alpha \text{ even}.$$
For $p |r,$
$$\theta_r(p^\alpha)  = 0\text{ for }\alpha \text{ odd};\ \  \theta_r(p^\alpha)  = p^{\alpha - 1}(p-1) \text{ for } \alpha \text{ even};$$
For $p = 2$, $$\theta_r(2^\alpha) := (-1)^\alpha 2^{\alpha - 1}.$$
For even $r$, $$\theta_r(2^\alpha)= 0$$ for any $\alpha \geq 1$; for an odd prime $p$, $\theta_r(p^\alpha)= \theta_{r'}(p^\alpha),$ where $r'$ is the odd part of $r$. 
\end{lemma}

\begin{proof}
The multiplicativity of the function follows immediately from the chinese remainder theorem and the multiplicativity of characters.

Suppose $m = p^\alpha$, $(p, 2Pr) = 1$, where $r$ can be even or odd. Then
\begin{multline*}
 \sum_{a \Mod m} \left(\frac{a}{m} \right) \left(\frac{r^2a - 4P}{m} \right) =  p^{\alpha - 1} \sum_{a \Mod p} \left(\frac{a}{p} \right)^\alpha \left(\frac{a - 4Pr^{-2}}{p} \right)^{\alpha}=  p^{\alpha - 1}  \sum_{a \in (\Z/p)^*}\left(\frac{1 - 4Pr^{-2} a^{-1}}{p} \right)^{\alpha}.
\end{multline*}
The value of $1 - 4Pr^{-2} a^{-1}$ for $a \in (\Z/p)^*$ is all the entries $\Mod p$ except for $1$. Hence, the sum is $-1$ for $\alpha$ odd and $p-2$ for $\alpha$ even.

Suppose now $p| r$. Then 
\[ \sum_{a \Mod m} \left(\frac{a}{m} \right) \left(\frac{r^2a - 4P}{m} \right) =  p^{\alpha - 1}  \left(\frac{-4P}{p} \right)^{\alpha} \sum_{a \Mod p} \left(\frac{a}{p} \right)^\alpha.\] The inner sum vanishes when $\alpha$ is odd and is $p-1$ when $\alpha$ is even.

Suppose now $p = 2$, $r$ is odd (so $\alpha \geq 3$).
\begin{multline*}
 \sum_{a \Mod m} \left(\frac{a}{m} \right) \left(\frac{r^2a - 4P}{m} \right) =2^{\alpha - 3} \sum_{a \Mod 8} \left(\frac{a}{2} \right)^{\alpha} \left(\frac{a - 4Pr^{-2}}{2} \right)^{\alpha}\\= 2^{\alpha - 3} \sum_{a \Mod 8} \left(\frac{a}{2} \right)^{\alpha} \left(- \frac{a}{2}\right)^{\alpha}= 2^{\alpha - 3} \sum_{a \in (\Z/8)^*}(-1)^\alpha.
\end{multline*}

Finally, if $p = 2$ and $2 | r$, the sum obviously vanishes.
\end{proof}

\begin{lemma}\label{tildephilemma}
Let $d, g, r \in \N$ be such that $g | d^\infty$. Then: 
$$\phi^o_{r, d}(g) = \begin{cases} \phi(g) \delta_{g = \square}&\text{ if } 2 \nmid d\\ \phi(g) \delta_{g = \square}&\text{ if } 2 || d, 2 \nmid g \\ 0 &\text{ if }  2||d, 2 | g, 2 ||r\\ 2 \phi(g)\delta_{g = \square} &\text{ if }  2|| d, 2 | g,4 |r \\  2 \phi(g) \delta_{g = \square} &\text{ if }  4 | d\end{cases}.$$
\end{lemma}
\begin{proof}
If $(d, r)$ is non-admissible, $\phi^o_{r, d}(g) = 0$, so assume it is an admissible pair. Let $t$ be an integer that reduces to an element of $\mathcal{R}_{r, d}$ modulo $d^2$, let $s:= (r^2 t-4P)/d^2$, and let $g =: 2^{\alpha }h$, where $(h, 2) = 1$. Since $h$ is odd, $\left(\frac{\cdot}{h}\right) $ is a character modulo the radical $\mathrm{rad} \ h$, and since $\mathrm{rad} \ h | d^2$, we have
\begin{align*}
\sum_{\substack{ a \Mod d^2 g \\ a \equiv t \Mod d^2}}  \left(\frac{a}{g}\right) \left(\frac{(r^2a-4P)/d^2}{g}\right) &=\sum_{\substack{ v = 1}}^g  \left(\frac{t + v d^2}{g}\right) \left(\frac{s + v r^2}{g}\right)\\
&= \sum_{v \Mod h}  \left(\frac{t}{h}\right) \left(\frac{s+ vt^2}{h}\right)  \sum_{t \Mod 2^\alpha} \left(\frac{t + v d^2}{2^
{\alpha}}\right) \left(\frac{s+ vt^2}{2^{\alpha}}\right)\\
&= \delta_{h = \square}\phi(h)\sum_{t \Mod 2^\alpha}  \left(\frac{t + v d^2}{2^{\alpha}}\right) \left(\frac{s+ vr^2}{2^{\alpha}}\right).
\end{align*} where in the last step we used that $(h, t) = 1$ because $(r, d) \leq 2$ by assumption on admissible pairs, and $h | d^\infty$.
We compute \[\star:= \sum_{t \Mod 2^\alpha}  \left(\frac{t + v d^2}{2^{\alpha}}\right) \left(\frac{s+ vr^2}{2^{\alpha}}\right)\] case by case.

\begin{enumerate}
\item If $\alpha = 0$, i.e., $2 \nmid g$, then $\star = 1$, so $\phi^o(g) = 2\delta_{g = \square} \phi(g)$ if $4 | d$ and $\delta_{g = \square} \phi(g)$ otherwise (because the sizes of $\mathcal{R}_{r, d}$ for such $d$ are $1$ and $2$, respectively).

\hspace{0.2in}  If $2 | g$ (and hence $2 | d$ and $2 | r$ from admissibility), then $t$ is odd since $(t, d) = 1$ for all $t \in \mathcal{R}_{d, r}$, and:
\item If $2 ||r, 2 ||d$, then $s$ is even and $\star = 0$.
\item If $4 | r$ and $2 ||d$, $\mathcal{R}_{r, d}$ has one element and the corresponding $s$ is odd, and so using that \[\left(\frac{x + 4}{2}\right) = -\left(\frac{x}{2}\right),\] we see  
\[\star =   \sum_{v \Mod 2^\alpha} \left(\frac{t + 4v}{2^\alpha}\right)  \left(\frac{s}{2^\alpha}\right) = \delta_{2^\alpha = \square} 2^\alpha = \delta_{2^\alpha = \square}2 \phi(2^\alpha);\]

\item  if $4 | d$ and $2 || r$, there's exactly one choice of $r \in \mathcal{R}_{r, d}$ for which $s$ is odd; for this choice of $t$, once again
$\star = \delta_{2^\alpha = \square} 2^\alpha;$ for the other choice of $t$, $s$ is even and $\star = 0$. 
\end{enumerate}
Combining all the cases, we get the statement of the lemma. 

Finally, we conclude by a routine computation evaluating the following sum:

\begin{lemma}\label{worstcomputation}
Let $\mathcal{T}_r \subset \N^3$ denote the set of triples $(m, d, g)$ such that $(d, r)$ is admissible, $(m, d) =  1$, and $g | d^\infty.$ Let
$$\Theta_r(m, d, g):=  \frac{\eta(d^2 m g)}{\phi(d^2 m g)}\frac{ \theta_r(m) \phi^o(g)}{mgd}.$$ Then $\sum_{\mathcal{T}_r} \Theta_r(m, d, g)$ is absolutely convergent and equal to \[B \cdot \nu(r):= \prod_p \frac{p^4 - 2 p^2 - p + 1}{(p^2-1)^2}  \prod_{p | r}\left( 1 + \frac{p^2}{p^4 - 2 p^2 - p + 1}\right).\] Moreover, 

$$\sum_{\substack{(m, d, g) \in \mathcal{T}_r \\ mg \leq Z',d \leq Z }} \Theta_r(m, d, g)  - B \nu(r) \ll Z^{-2} + (Z')^{-1/5} .$$

\end{lemma}

\begin{proof}
Define $E:= \prod_{p} \left(1 +\frac{p}{(p^2-1)^2} \right), E(r):=  \prod_{p | r} \left(1 +\frac{p}{(p^2-1)^2} \right),$ $G:= \prod_{p}\left(1 - \frac{2p}{(p^2 - 1)^2 + p}\right), G(r):= \prod_{p | r}\left(1 - \frac{2p}{(p^2 - 1)^2 + p}\right),$ and $F(r) := \prod_{p | r} \left( 1 + \frac{p(p-1)}{(p^2-1)^2}\right).$
It is easy to verify that \[B \nu(r) =  \frac{EG F(r)}{E(r) G(r)}.\]
Next, since $g | d^\infty$ and $(m, d) = 1$, we have $\{ p | md^2g\} = \{p | d\} \sqcup \{ p |m \},$ so
\begin{align*}
\Theta_r(m, d, g) = \frac{1}{d^3  \prod_{p | d} (1 - 1/p^2)}  \frac{\theta_r(m)}{m^2\prod_{p | m} (1 - 1/p^2)} \frac{\phi^o(g)}{g^2}.
\end{align*}
We begin by establishing absolute convergence. For fixed $m, d$, the sum
\[\sideset{}{^g}\sum := \sum_{g | d^\infty} \frac{\phi^o(g)}{g^2}\] over $g$'s appearing in $\mathcal{T}_r$ for these $m$ and $d$ is upper bounded by
\begin{multline*}
\sum_{\substack{ g | d^\infty \\ g= \square}} \frac{2\phi(g)}{g^2} = 2\prod_{p | d}\left( 1 + \sum_{k \geq 1} \frac{p-1}{p} \frac{1}{p^{2r}}\right) \\
= 2 \prod_{p | d}\left(1 + \frac{1}{p(p + 1)}\right) \ll \prod_p(1 + 1/p^2) \leq \prod_p 1/(1 - 1/p^2) = \zeta(2),
\end{multline*} so is uniformly bounded. 

Now fix $d$. Any number $r$ can be written uniquely as $a^2\cdot  b \cdot 2^c$, where $b$ is square-free and $a, b$ odd. From the definition of $\theta_r$, $$\abs{\theta_r(a^2\cdot b \cdot2^c)} \leq a^2 \cdot 2^c.$$ In particular, using that $\prod_{p | m} (1 - 1/p^2) \gg 1$, 
$$\sum_{(m, d) = 1}  \frac{\theta_r(m)}{ m^2 \prod_{p | m} (1 - 1/p^2)}  \ll \sum_{m \in \N} \abs{\theta_r(m)}/m^{2} \ll \left(\sum_{c \in \N} 1/2^{c}\right)\left(\sum_{a \in \N}  1/a^{2} \right)\left( \sum_{b \sqf}  1/b^{2 }\right)\ll 1.$$

Finally, \[ \sum_{d \in \N} \frac{1}{d^3  \prod_{p | d} (1 - 1/p^2)}  \ll 1,\] so indeed, the series $\sum_{\mathcal{T}_r} \Theta_r(m, d, g)$ converges absolutely. 

 Next, we compute $\sideset{}{^g}\sum$ case by case from the definition.
\begin{itemize}
\item When $2 \nmid d$, $\phi^o(g) =\delta_{g = \square} \phi(g),$ and \[\sideset{}{^g}\sum = \sum_{\substack{ g | d^\infty \\ g= \square}} \frac{\phi(g)}{g^2}  = \prod_{p | d}\left(1 + \frac{1}{p(p + 1)}\right);  \]
\item When $2 || d$ and $2 ||r$, $\phi^o(g) =\delta_{g = \square} \phi(g)$ for odd $g$ and $0$ for even $g$, so \[\sideset{}{^g}\sum = \sum_{\substack{ g | d^\infty \\ g= \square \\ 2 \nmid g}} \frac{\phi(g)}{g^2}  =\prod_{p | d_o}\left(1 + \frac{1}{p(p + 1)}\right)\] where $d_o$ is the odd part of $d$; 
\item When $2 || d$ and $4 | r$, then  $\phi^o(g) =\delta_{g = \square} \phi(g)$ for $g$ odd and  $\phi^o(g) =2\delta_{g = \square} \phi(g)$ for $g$ even, so 
 \[\sideset{}{^g}\sum = \sum_{\substack{ g | d_0^\infty \\ g= \square }} \frac{\phi(g)}{g^2} + 2  \sum_{\substack{ g | d^\infty \\ g= \square \\ 2 | g}} \frac{\phi(g)}{g^2}  = \frac{4}{3} \prod_{p | d_o}\left(1 + \frac{1}{p(p + 1)}\right);\]
\item When $4 | d$ and $2 || r$, then  $\phi^o(g) =2 \delta_{g = \square} \phi(g)$, so 
 \[\sideset{}{^g}\sum =2  \prod_{p | d}\left(1 + \frac{1}{p(p + 1)}\right) = \frac{7}{3}  \prod_{p | d_o}\left(1 + \frac{1}{p(p + 1)}\right) ;\]
\end{itemize}

We can now compute the desired sum by expressing it as an Euler product. 

\textbf{Case I: $2 \nmid r$.}

Suppose first that $r$ is odd, so $(r, d)$ is admissible if and only if $d$ is odd and $(r, d) = 1$. Then
\begin{align*}
 \sum_{\mathcal{T}_r} \Theta_r(m, d, g) &= \sum_{\substack{ d \text{ odd}\\ (d, r) = 1}} \sum_{(m, d) = 1}   \frac{\prod_{p | d} (1 + 1/p(p+1))}{d^3  \prod_{p | d} (1 - 1/p^2)}  \frac{\theta_r(m)}{m^2\prod_{p | m} (1 - 1/p^2)}  \\
&=\sum_{m} \frac{\theta_r(m)}{m^2\prod_{p | m} (1 - 1/p^2)} \sum_{(d, 2mr) = 1}\frac{1}{d^3} \prod_{p | d} \frac{p(1 + p + p^2)}{(p-1)(p + 1)^2}.\\
\end{align*}
The inner sum can be expressed as the Euler product
\begin{multline*}
\prod_{p \nmid 2mr} \left( 1 + \sum_{k \geq 1} \frac{1}{p^{3k}} \frac{p(1 + p + p^2)}{(p-1)(p + 1)^2} \right) \hspace{-0.05in}= \hspace{-0.05in} \prod_{p \nmid 2mr} \left( 1 +\frac{p}{(p^2 -1)^2} \right) \\
= \prod_{p} \left( 1 +\frac{p}{(p^2 -1)^2} \right)/\prod_{p | 2mr} \left( 1 +\frac{p}{(p^2 -1)^2} \right)
\end{multline*}
so 
$$ \sum_{\mathcal{T}_r} \Theta_r(m, d, g) = \frac{E}{E(r)} \cdot \sum_m \frac{\theta_r(m)}{m^2}  \prod_{p | m}(1 - 1/p^2)^{-1} \prod_{p | 2m, p \nmid k} (1 + p/(p^2 - 1)^2)^{-1}.$$
This sum over $m$ can itself be expressed as an Euler product. The $p \neq 2$, $p \nmid r$ part of the product is
\begin{align*}
&\prod_{p \nmid 2r}\left(1 + \frac{1}{(1 - 1/p^2)(1 + p/(p^2 - 1)^2)}\left( \sum_{\alpha \geq 0} \frac{\theta_r(p^{2\alpha + 1})}{p^{4\alpha + 2}} + \sum_{\alpha \geq 0} \frac{\theta(p^{2\alpha + 2})}{p^{4\alpha + 4}} \right) \right)\\
= &\prod_{p \nmid 2r }\left(1 + \frac{1}{(p^2 - 1)(1 + p/(p^2 - 1)^2)}\left( - \sum_{\alpha \geq 0} \frac{1}{p^{2\alpha}} + \sum_{\alpha \geq 0} \frac{(p-2)}{p^{2\alpha + 1}} \right) \right)\\
= &\prod_{p \nmid 2r }\left(1 - \frac{2p}{(p^2 - 1)^2 + p}\right) = \frac{G}{G(2r)}.
\end{align*}
The $p |r$ factors are
\begin{align*}
\prod_{p | r}\left(1 + \frac{1}{1 - 1/p^2}\left( \sum_{\alpha \geq 0} \frac{\theta_r(p^{2\alpha + 1})}{p^{4\alpha + 2}} + \sum_{\alpha \geq 0} \frac{\theta(p^{2\alpha + 2})}{p^{4\alpha + 4}} \right) \right)&= \prod_{p | r}\left(1 + \frac{1}{p^2 - 1}\sum_{\alpha \geq 0} \frac{p-1}{p^{2\alpha + 1}}  \right)\\
&= \prod_{p | r}\left(1 + \frac{p(p-1)}{(p^2-1)^2}\right) = F(r).
\end{align*}
The Euler factor at $2$ is 
$$ \frac{1}{1 + 2/(2^2 - 1)^2} \left(1 + \sum_{\alpha \geq 1} \left(1 - \frac{1}{2^2} \right)^{-1}\frac{(-1)^\alpha 2^{\alpha - 1}}{2^{2\alpha}} \right) =  \frac{9}{11} \left(1 + \frac{1}{2} \frac{4}{3} \sum_{\alpha \geq 1}(-1/2)^\alpha\right) = \frac{7}{11} = G(2).$$
All the Euler factors together add up to $B \nu(r)$.

Now we estimate the tail, namely, the contribution of terms with $d > Z$ or $mg > W$.  
As noted above, for $d$ fixed, 
$$\sum_{(m, d) = 1} \sum_{g | d^\infty,g = \square}  \frac{\theta_r(m)}{m^2\prod_{p | m} (1 - 1/p^2)} \frac{\phi(g)}{g^2} $$ is uniformly bounded; hence, 

\begin{align}\label{dbd}
  \sum_{(m, d, g) \in \mathcal{T}_r: d \geq Z} \Theta_r(m, d, g) \ll \sum_{d \geq Z}  \frac{1}{d^3  \prod_{p | d} (1 - 1/p^2)}  \ll Z^{-2}.
\end{align}
For a given $g$, the contribution of summands with that $g$ is 
\begin{multline*}
\frac{\phi(g)}{g^2}\sum_{\substack{m, d\\ (m, d) = 1\\d \text{ odd} \\ g | d^\infty }} \frac{1/d^3}{\prod_{p | d} (1 - 1/p^2)}  \frac{\theta_r(m)/m^2}{\prod_{p | m} (1 - 1/p^2)} \\
\leq \frac{\phi(g)}{g^2} \sum_{\substack{m, d\\ (m, d) = 1\\d \text{ odd}  }} \frac{1/d^2}{ \prod_{p | d} (1 - 1/p^2)}  \frac{|\theta_r(m)|/m^2}{\prod_{p | m} (1 - 1/p^2)} \ll \frac{\phi(g)}{g^2}.
\end{multline*}
 Here we used that $\sum_{\substack{m, d\\ (m, d) = 1\\d \text{ odd}  }} \frac{1/d^2}{ \prod_{p | d} (1 - 1/p^2)}  \frac{|\theta_r(m)|/m^2}{\prod_{p | m} (1 - 1/p^2)}$ converges since these are terms of the original sum corresponding to $g = 1$. From this,
\begin{align}\label{gbd}
\sum_{(m, d, g) \in \mathcal{T}_r: g \geq W} \Theta_r(m, d, g) \ll\sum_{t > \sqrt{W}} \frac{1}{t^2}\ll W^{-1/2}.
\end{align}
Finally, the terms in the sum corresponding to elements of $\mathcal{T}_r$ with a fixed $m$ are 

\begin{multline*}
  \frac{\theta_r(m)/m^2}{\prod_{p | m} (1 - 1/p^2)} \sum_{\substack{g, d\\ (m, d) = 1\\d \text{ odd} \\ g | d^\infty }} \frac{\phi(g)}{g^2} \frac{1/d^3}{\prod_{p | d} (1 - 1/p^2)} \\ \leq  \frac{\theta_r(m)/m^2}{\prod_{p | m} (1 - 1/p^2)}   \sum_{\substack{g, d\\d \text{ odd} \\ g | d^\infty }} \frac{\phi(g)}{g^2} \frac{1/d^3}{\prod_{p | d} (1 - 1/p^2)}  \ll   \frac{\theta_r(m)/m^2}{\prod_{p | m} (1 - 1/p^2)},
\end{multline*}
where again, $ \sum_{\substack{g, d\\d \text{ odd} \\ g | d^\infty }} \frac{\phi(g)}{g^2} \frac{1/d^3}{\prod_{p | d} (1 - 1/p^2)}$ converges since these are the $m = 1$ terms in the original sum. Thus, 
$$\sum_{(m, d, g) \in \mathcal{T}_r: m \geq V} \Theta_r(m, d, g) \ll  \sum_{m \geq V} |\theta_r(m)|/m^2.$$ Recall that for $m = a^2\cdot  b\cdot 2^c$ for $b$ square-free, $a, b$ odd, one has $\theta_r(m)/m^2 \ll\frac{1}{a^2 b^2 2^c}.$ Hence

\begin{multline}\label{mbd}\sum_{m \geq V} |\theta_r(m)|/m^2 \leq \sum_{c \geq 1} \frac{1}{2^c} \sum_{\substack{ m = a^2 b  \text{ odd }\\ m\geq V/2^c}} |\theta_r(m)|/m^2 \\
\ll   \sum_{c \geq 1} \frac{1}{2^c} \Big(\sum_{a >(V/2^c)^{1/3}}\sum_b \frac{1}{a^2b^2}+\sum_{b >(V/2^c)^{1/3}}  \sum_a \frac{1}{a^2b^2} \Big) \ll \sum_c \frac{1}{2^c} \frac{2^{c/3}}{V^{1/3}} \ll V^{-1/3}. 
\end{multline}

Finally, let $Z' = VW$, where $V = Z^{3/5}$, $V = Z^{2/5}$. Then $gm > Z'$ implies that at least one of $g > W$ or $m >V$ holds, so putting together (\ref*{dbd}), (\ref*{gbd}), and (\ref*{mbd}), 

$$\sum_{\substack{(m, d, g) \in \mathcal{T}_r \\ mg \geq Z' \text{ or } d \geq Z }} \Theta_r(m, d, g)  \ll Z^{-2} + (Z')^{-1/5}.$$

\textbf{Case II: $2 || r$.}

Next, suppose $r$ is even, let $r_o$ the odd part of $r$.
\begin{align*}
 \sum_{\mathcal{T}_r} \Theta_r(m, d, g) &= \sum_{\substack{ d \text{ odd }\\ (d, r_o) = 1}} \sum_{(m, 2d) = 1}   \frac{\prod_{p | d_o} (1 + 1/p(p+1))}{d^3  \prod_{p | d} (1 - 1/p^2)}  \frac{\theta_{r_o}(m)}{m^2\prod_{p | m} (1 - 1/p^2)}  \\
& +\sum_{\substack{ 2||d \\ (d, r_o) = 1}} \sum_{(m, d) = 1}   \frac{\prod_{p | d_o} (1 + 1/p(p+1))}{d^3  \prod_{p | d} (1 - 1/p^2)}  \frac{\theta_{r_o}(m)}{m^2\prod_{p | m} (1 - 1/p^2)}  \\
& +  \frac{7}{3} \sum_{\substack{ 4|d \\ (d, r_o) = 1}} \sum_{(m, d) = 1}   \frac{\prod_{p | d_o} (1 + 1/p(p+1))}{d^3  \prod_{p | d} (1 - 1/p^2)}  \frac{\theta_{r_o}(m)}{m^2\prod_{p | m} (1 - 1/p^2)}\\ &=: I + II + (7/3)III.
\end{align*}
We compute the three summands $I, II,$ and $III$, separately. 

Observe that for $m$ odd and $\alpha \geq 1$,  
\begin{multline*}
 \frac{\theta_{r_o}(2^\alpha m)}{2^{2 \alpha}m^2\prod_{p | 2m} (1 - 1/p^2)} = \frac{1}{2^{2\alpha}} \frac{1}{1 - 1/4} (-1)^\alpha 2^{\alpha - 1} \frac{\theta_{r_o}(m)}{m^2\prod_{p | m} (1 - 1/p^2)} \\
=  \frac{2 \cdot (-1)^\alpha}{3 \cdot 2^{\alpha}}  \frac{\theta_{r_o}(m)}{m^2\prod_{p | m} (1 - 1/p^2)}.
\end{multline*}

Thus 

\begin{multline*}
I  =  \sum_{\substack{ d \text{ odd }\\ (d, r_o) = 1}} \sum_{(m, d) = 1}   \frac{\prod_{p | d} (1 + 1/p(p+1))}{d^3  \prod_{p | d} (1 - 1/p^2)}  \frac{\theta_{r_o}(m)}{m^2\prod_{p | m} (1 - 1/p^2)} \left( 1 + \sum_{\alpha \geq 1} \frac{2 \cdot (-1)^\alpha}{3 \cdot 2^{\alpha}} \right)^{-1}\\
 = \frac{9}{7} \sum_{\mathcal{T}_{r_o}} \Theta_{r_o}(m, d, g).  
\end{multline*}

For $II$, we can rewrite
\begin{multline*}
II =  \sum_{\substack{ d_o \text{ odd } \\ (d, r_o) = 1}}   \frac{1}{8 (1 - 1/4)} \frac{\prod_{p | } (1 + 1/p(p+1))}{d^3  \prod_{p |d } (1 - 1/p^2)} \sum_{(m, 2d_o) = 1}  \frac{\theta_{r_o}(m)}{m^2\prod_{p | m} (1 - 1/p^2)}\\
 =  \sum_{\substack{ d \text{ odd } \\ (d, r_o) = 1}}   \frac{9}{7} \frac{1}{6} \sum_{\mathcal{T}_{r_o}} \Theta_{r_o}(m, d, g).
\end{multline*}
Finally, 
\begin{multline*}
III = \sum_{\alpha \geq 2} \sum_{\substack{ d_o \text{ odd } \\ (d, r_o) = 1}}   \frac{1}{8^\alpha (1 - 1/4)} \frac{\prod_{p | d} (1 + 1/p(p+1))}{d^3  \prod_{p | d} (1 - 1/p^2)} \sum_{(m, 2d) = 1}  \frac{\theta_{r_o}(m)}{m^2\prod_{p | m} (1 - 1/p^2)} \\
= \frac{9}{7} \frac{1}{6 \cdot 7}  \sum_{\mathcal{T}_{r_o}} \Theta_{r_o}(m, d, g).
\end{multline*}

In summary, 
\[ \frac{\theta_{r_o}(2^\alpha m)}{2^{2 \alpha}m^2\prod_{p | 2m} (1 - 1/p^2)} = \frac{11}{9}\left(1 + \frac{1}{6} + \frac{1}{6 \cdot 7} \frac{7}{3}\right)  \sum_{\mathcal{T}_{r_o}} \Theta_{r_o}(m, d, g) = \frac{11}{7}  \sum_{\mathcal{T}_{r_o}} \Theta_{r_o}(m, d, g).\]

It remains to notice that $\nu(r) = 11/7 c(r_o)$. 

\textbf{Case III: $4 || r$.}

Finally, we address the case $4 | r$: 
\begin{align*}
 \sum_{\mathcal{T}_r} \Theta_r(m, d, g) &= \sum_{\substack{ d \text{ odd }\\ (d, r_o) = 1}} \sum_{(m, 2d) = 1}   \frac{\prod_{p | d_o} (1 + 1/p(p+1))}{d^3  \prod_{p | d} (1 - 1/p^2)}  \frac{\theta_{r_o}(m)}{m^2\prod_{p | m} (1 - 1/p^2)}  \\
& + \frac{4}{3}\sum_{\substack{ 2||d \\ (d, r_o) = 1}} \sum_{(m, d) = 1}   \frac{\prod_{p | d_o} (1 + 1/p(p+1))}{d^3  \prod_{p | d} (1 - 1/p^2)}  \frac{\theta_{r_o}(m)}{m^2\prod_{p | m} (1 - 1/p^2)}.
\end{align*}
Since the first summand matches $I = \frac{9}{7} \cdot c(r_o)$; the second is equal to $\frac{4}{3} \cdot II = \frac{4}{3} \frac{1}{6} \frac{9}{7} \cdot c(r_o) = \frac{2}{7} c(r_o),$ and we get the desired answer once again.

The error term analysis for $r$ even matches that of odd $r$.

\end{proof}

\subsubsection{Truncation}\label{truncation}
Let $0\leq \tau < 1/2$ be a parameter to be determined later. From the analysis of the previous subsection and the class number trivial upper bound,
\begin{multline}\label{coup}
 \sideset{}{^\square}\sum_{\substack{ N \in [X, X + Y] \\ P \nmid N }} H_1(r^2 N^2 - 4 PN) =  \sum_{ d^2 < Y^\tau}\ \sideset{}{^\square}\sum_{\substack{ N \in [X, X + Y]  \\ N \in {\mathcal{A}}_{r, d} \\ P \nmid N }} \frac{L(1, \chi_{\frac{r^2N^2 - 4PN}{d^2}})}{\pi d}  \sqrt{4 P N - r^2N^2} + \\
+ \O\bigg(\sum_{\sqrt{P} \gg d \gg Y^\tau} \left(\frac{Y}{d^2} + 1 \right)\frac{(PX)^{\frac{1}{2} + \eps}}{d} \bigg).
\end{multline}
Note that for $N \in [X, X + Y]$ (and assuming $r^2 (X + Y) < 4 P$),

\begin{align*}
\sqrt{4 P N- r^2 N^2} &- \sqrt{4 P X - r^2 X^2} \\
&=  \sqrt{4P - r^2 N}\left( \sqrt{N} - \sqrt{X}\right)  -  \sqrt{4 P X - r^2 X^2}\left( 1 - \sqrt{1 - \frac{r^2 (N - X)}{4P - r^2 X}}\right)\\
&\ll \sqrt{PX} \left( \sqrt{1 + \frac{Y}{X}} - 1\right) - {\sqrt{X}  \sqrt{4 P  - r^2 X}\left( 1 - \sqrt{1 - \frac{r^2 Y}{4P - r^2 X}}\right)}\\  &\ll Y P^{\frac{1}{2}} X^{-\frac{1}{2}} + r (XY)^{\frac{1}{2}}.
\end{align*}
Combining this with the trivial upper bound on the special value of the $L$ function,
\begin{multline*}
(\ref*{coup}) = \sum_{d < Y^\tau }\ \sideset{}{^\square} \sum_{\substack{ N \in [X, X + Y]  \\ N \in {\mathcal{A}}_{r, d}, P \nmid N }} \frac{L(1, \chi_{\frac{r^2N^2 - 4PN}{d^2}})}{\pi d}  \sqrt{4 P X - r^2X^2} + \\  + \O\big((PX)^\eps (Y^2 P^{\frac{1}{2}} X^{-\frac{1}{2}} + r Y^{\frac{3}{2}} X^{\frac{1}{2}} + (PX)^\frac{1}{2}Y^{1 - 2 \tau}) \big).
\end{multline*}
By (\ref*{truncate}), we can truncate the $L$ function in the main term above as
\begin{align*} \sqrt{4 P X - r^2X^2} \sum_{d < Y^\tau }\frac{1}{\pi d} \ \sideset{}{^\square}  \sum_{\substack{ N \in [X, X + Y]  \\ N \in {\mathcal{A}}_{r, d} \\ P \nmid N }} \sum_{n = 1}^T \frac{\left(\frac{(r^2N^2 - 4PN)/d^2}{n}\right)}{n}  + \O\Big( \sqrt{PX}\sum_{d < Y^\tau} \sum_{X}^{X + Y}  \frac{(PX)^{\frac{1}{2} + \eps}}{d T}\Big),
\end{align*}
we conclude that
\begin{multline}\label{identity} 
 \sideset{}{^\square}\sum_{\substack{ N \in [X, X + Y] \\ P \nmid N }} H_1(r^2 N^2 - 4 PN) = \frac{ \sqrt{4 P X - r^2X^2}}{\pi }  \sum_{d^2 \leq Y^\tau }\sum_{n \leq T}  \frac{\mathcal{S}_{d, n, r}}{nd}+ \\
+  \O\Big(\Big( \frac{YPX}{T}+ \frac{Y^2 P^{\frac{1}{2}}}{ X^{\frac{1}{2}}} + r Y^{\frac{3}{2}} X^{\frac{1}{2}} +  (PX)^\frac{1}{2}Y^{1 - 2 \tau}\Big)(PX)^\eps\Big) ,
\end{multline}
where we define $\mathcal{S}_{d, n, r}$ as follows: 

\begin{defn}
Let $d, r$ be positive integers, let $Y = o(X)$, and let $P$ be a prime with that $4P > r^2(X + Y)$. Define 
\[\mathcal{S}_{d, n, r} := \sum_{\substack{ N \in [X, X + Y]\\N \in \mathcal{A}_{r, d} \\ P \nmid N}}\mu^2(N)\left(\frac{N}{n}\right)  \left(\frac{(r^2N-4P)/d^2}{n}\right).\]
Note that $\mathcal{S}_{d, n, r} = 0$ unless $(r, d)$ is an admissible pair.
\end{defn}

\subsubsection{Small $n$}\label{smalln}
In this subsection we analyze the case $n < Y^\sigma$. We prove: 

\begin{prop}\label{smalldn}
For any parameters $0 \leq \tau, \sigma \leq 1$,

\[\sum_{n \leq Y^\sigma} \sum_{d < Y^\tau}\frac{\mathcal{S}_{d, n, r}}{nd}= \frac{Y B  \nu(r)}{\zeta(2)} + \O\big( \sqrt{X} Y^{\frac{\sigma}{2}} + Y^{\frac{3 \sigma}{2} + \tau + \eps} + Y^{1 - 2 \tau} + Y^{1 -\frac{ \sigma}{5}} \big).\]
\end{prop}

The function $\mathcal{S}_{d, n, r}$ can be approximated by multiplicative functions defined in Definition \ref*{multfuncdefn} via the next two lemmas.
\begin{lemma}\label{weirdsum} Let $m$ be an integer with $P \nmid m$ and $\upsilon_2(m) \neq 1, 2$, and let $(r, d)$ be an admissible pair. Then: 
\[ \sum_{\substack{ a \Mod d^2 m \\ a \in \mathcal{R}_{r, d} \Mod d^2}} \left(\frac{a}{m}\right)  \left(\frac{(r^2a-4P)/d^2}{m}\right) = \phi^o_{r, d}(g) \theta_r(m'),\] where $g:= (d^\infty, m)$, and $m':= m/g$. 
\end{lemma}
\begin{proof}
 Let $t$ be an integer that reduces to an element of $\mathcal{R}_{d, r}$ modulo $d^2$. Then: 
\begin{align*}
 \sum_{\substack{ a \Mod d^2 m \\ a \equiv t \Mod d^2}} \left(\frac{a}{m}\right)  \left(\frac{(r^2a-4P)/d^2}{m}\right)  &=  \sum_{\substack{ a \Mod d^2 m \\ a \equiv t \Mod d^2}} \left(\frac{a}{m'}\right)  \left(\frac{(r^2a-4P)/d^2}{m'}\right)  \left(\frac{a}{g}\right)  \left(\frac{(r^2a-4P)/d^2}{g}\right) \\
&=  \sum_{b \Mod m'} \left(\frac{b}{m'}\right)  \left(\frac{r^2b-4P}{m'}\right)  \sum_{\substack{ a \Mod d^2 g \\ a \equiv t \Mod d^2}}  \left(\frac{a}{g}\right) \left(\frac{(r^2a-4P)/d^2}{g}\right)\\
&= \theta_r(m') \sum_{\substack{ a \Mod d^2 g \\ a \equiv t \Mod d^2}}  \left(\frac{a}{g}\right) \left(\frac{(r^2a-4P)/d^2}{g}\right)
\end{align*}
where we applied Lemma \ref*{thetalemma} and the assumption $P \nmid m'$ in the last step. It remains to sum the above identity over $r \in \mathcal{R}_{r, d}$, which is $\phi^o_{d, r}(g)$ by definition. 

\end{proof}
\begin{lemma}\label{silly}
Let $d, n, r \geq 1$ be integers with $P \nmid n$, such that $(d, r)$ is an admissible pair. 
Then 
\[\mathcal{S}_{d, n, r}  =  \frac{Y}{\zeta(2)} \frac{\eta(d^2 n)}{\phi(d^2n)}  \phi^o_{r, d}(g)  \theta_r(n') + \O\left( \sqrt{Xn}/d + d n^{3/2 + \eps}\right),\] where $g:= (d^\infty, n)$, and $n':= n/g$. 
\end{lemma}
\begin{proof}
Let $t \in \mathcal{R}_{r, d}$ be a remainder modulo $d^2$. The character $\left(\frac{(r^2x-4P)/d^2}{n}\right)$ is a function of $x \Mod f n d^2$, where $f = 4$ if $n$ is even and $f = 1$ if $n$ is odd. Thus,
\begin{align}\label{images}
 \sum_{\substack{ N \in [X, X + Y]\\N \equiv t \Mod d^2 \\ P \nmid N}}\mu^2(N)\left(\frac{N}{n}\right)  \left(\frac{(r^2N-4P)/d^2}{n}\right) &=  \hspace{-0.1in} \sum_{\substack{ a \Mod f d^2 n \\ a \equiv t \Mod d^2}} \sum_{\substack{ N \in [X, X + Y]\\N \equiv a \Mod f d^2n \\ P \nmid N}}\hspace{-0.15in}\mu^2(N)\left(\frac{N}{n}\right)  \left(\frac{(r^2N-4P)/d^2}{n}\right) \nonumber\\
&= \hspace{-0.1in}\sum_{\substack{ a \Mod f d^2 n \\ a \equiv t \Mod d^2}} \left(\frac{a}{n}\right)  \left(\frac{(r^2a-4P)/d^2}{n}\right)  \sum_{\substack{ N \in [X, X + Y]\\N \equiv a \Mod f d^2n \\ P \nmid N}}\hspace{-0.15in}\mu^2(N).
\end{align}
Since the class numbers $H(r^2 N^2 - 4 P N)$ with $r \geq 1$ appear only for $P \geq N/4 \geq X/4$, we can remove the condition $P \nmid N$ in the expression above for a cumulative error term of $\O(n)$. Evidently, the terms with $(a, n) > 1$ above vanish; furthermore, $(r, d) = 1$ for all $r \in \mathcal{R}_{r, d}$, so $(a, d) = 1$; hence it suffices to consider $a$ coprime to $f d^2 n$. For such $a$, we can apply a Theorem of Hooley: 
\begin{theorem}[Hooley, \cite{hooley}]\label{idealtheorem}:
Let $(a, m) = 1$.
$$\sum_{\substack{N =X \\ N = a \Mod{m}}}^{X +Y} \mu^2(N) = \frac{Y}{\zeta(2)} \frac{\eta(m)}{\phi(m)} + O(\sqrt{X/m} + m^{1/2 + \eps}).$$ 
\end{theorem}
From this, 

\[(\ref*{images}) = \frac{Y}{\zeta(2)} \frac{\eta(d^2 n)}{f \phi(d^2n)}  \sum_{\substack{ a \Mod d^2 f n \\ a \equiv t \Mod d^2}} \left(\frac{a}{n}\right)  \left(\frac{(r^2a-4P)/d^2}{n}\right) + \O\left( \sqrt{Xn}/d + d^{1 +\eps} n^{3/2 + \eps}\right).\]
The Lemma then follows from Lemma \ref*{weirdsum} applied to $m = fn $ along with the fact that $\phi^o(fn) = f \phi^o(n)$ (which follows from Lemma \ref*{tildephilemma}).
\end{proof}

\begin{proof}[Proof of Proposition \ref*{smalldn}]
Since $P \gg X$, $Y = o(X)$, and $n \ll Y^\sigma$, the condition $(P, n) = 1$ holds asymptotically. Thus, by Lemma \ref*{silly} and Lemma \ref*{worstcomputation}, 
\begin{align*}
\sum_{\substack{ n \leq Y^\sigma\\ d \leq Y^\tau}}  \frac{\mathcal{S}_{d, n, r}}{nd} &= \hspace{-0.2in}\sum_{\substack{  n \leq Y^\sigma\\   d \leq Y^\tau\\ (r, d) \text{ admissible }}} \hspace{-0.2in}  \frac{Y}{\zeta(2)} \frac{\eta(d^2n)}{\phi(d^2n)}\frac{\phi^o_{r, d}(g)\theta_r(n')}{nd}  +  \sum_{\substack{ n \leq Y^\sigma\\ d \leq Y^\tau}}  \O\left( \frac{\sqrt{Xn}}{nd^2 } + \frac{d^{1 + \eps} n^{3/2 + \eps}}{nd} \right)\\
 &= \frac{Y B \nu(r)}{\zeta(2)} + \O\left( \sqrt{X} Y^{\frac{\sigma}{2}} + Y^{\frac{3 \sigma}{2} + \tau + \eps} + Y^{1 - 2 \tau} + Y^{1 - \frac{\sigma}{5}} \right)
\end{align*}
as aimed. 
\end{proof}
\subsubsection{Large $n$.}\label{largen}
In this subsection we prove:
\begin{prop}\label{smalldlargem}
For any parameters $0 \leq \tau, \sigma \leq 1$,
\[\sum_{n \in [Y^\sigma, T]} \sum_{d < Y^\tau} \frac{\mathcal{S}_{d, n, r}}{nd} \ll  (TPX)^\eps \left( \sqrt{X} + Y^{1 - \frac{\sigma}{2}} + \sqrt{Y} T^{\frac{1}{4}}\right).\]
\end{prop}

\begin{lemma}\label{shortintchar}
Let  $\pi(x): \Z \to S^1 \subset \C$ be an $m$-periodic function from the integers to the unit circle, and let $\ell \in \R_+$. Then:
$$\max_{I : \abs{I} = \ell}\Big| \sum_{x \in I} \pi(x)\Big|\ll  \sqrt{\mathcal{M} \ell} + \frac{\mathcal{M} \ell}{m} + 1,$$
where the maximum is taken over all intervals $I \subset \R$ of length $\ell$ and $$\mathcal{M}:= \max_{n \Mod m} \Big|\sum_{b \Mod m}  e(b n/m) \pi(b) \Big| \text{ for } e(x) := e^{2 \pi i x}.$$
\end{lemma}
\begin{proof}
Let $I = [c, c + \ell]$ and let $\eps = \eps_\ell$ be a parameter to be chosen later. For an $m$-periodic function $\pi$ and a Schwartz function $f$, one has by Poisson summation that 

\begin{align}\label{arithmeticPoisson2}
\sum_{n \in \Z} \pi(n) f(n) &= \sum_{b \Mod m}\pi(b) \sum_{a \in \Z} f(m a + b) = \sum_{b \Mod m} \frac{\pi(b)}{m} \sum_{n \in \Z} e\left( b n/m \right) \widehat{f}\left( \frac{n}{m}\right).
\end{align}
We construct a test function $f$ as follows.
Let $\psi(x) : \R \to \R_+$ be smooth function with $\norm{\psi}_{1} = 1$ and $\mathrm{supp}(\psi) \subset [-1, 1].$ Define
$$\psi_\eps(x):= \frac{\psi(x/\eps)}{\eps}, \widehat{\psi_\eps}(t) = \widehat{\psi}(\eps t).$$ Then $\norm{\psi_\eps}_{1}  = 1$ and $\norm{\smash{\widehat{\psi_\eps}}}_{\infty} = \norm{\smash{\widehat{\psi}}}_{\infty} \leq 1,$ and  $\mathrm{supp}(\psi) \subset [-\eps, \eps]$. Let $$f_{c, \eps}(t) := (\chi_{[0, 1]} * \psi_\eps) \left(\frac{t - c}{\ell} \right),$$ where $\chi_{[0, 1]}$ is the characteristic function of the interval $[0, 1]$. The function $f = f_{c, \eps}$ is a smooth function satisfying the following properties: 

\begin{enumerate}
\item $f(t) \in [0, 1]$;
\item $\mathrm{supp} (f) \subseteq [c - \eps \ell, c + \ell + \eps \ell]$;
\item $f(t) = 1$ for $t \in [c + \ell \eps; c + \ell - \ell \eps];$
\item $|\widehat{f}(t)| = \ell \cdot \abs{\widehat{\chi}_{[0, 1]}(t\ell)} \cdot \big|\widehat{\psi}(\eps t \ell)\big|$ \\
\item $\widehat{f}(t) \leq \ell, \widehat{f}(t) \ll_{K, \psi} \dfrac{\ell}{(\eps \ell t)^K}$ for $K \geq 2$ (since $\widehat{\psi}$ is Schwartz).
\end{enumerate}
 Properties $1, 2,$ and $3$ imply that 
$$ \sum_{x \in I} \pi(x)=  \sum_{n \in \Z} \pi(n) f(n) + \O\left( \eps \ell + 1 \right).$$
From (\ref*{arithmeticPoisson2}),
\begin{multline*}
\sum_{n \in \Z}\pi(n)  f(n) = \frac{1}{m} \sum_{n\in \Z} \widehat{f} \left(\frac{n}{m}\right) \sum_{b \Mod m} e(b n/m) \pi(b) \\ \ll \frac{\max_{n \Mod m} \abs{\sum_{b \Mod m}  e(b n/m) \pi(b) }}{m} \sum_{n \in \Z} \abs{\widehat{f}(n/m)}.
\end{multline*}
Choosing an integer parameter $X_c = \max\{ 1, m/\eps \ell \}$ and using property $5$, we thus get a bound
\begin{multline*}
\sum_{x \in I} \pi(x) = \sum_{n \in \Z} \pi(n) f(n) + \O(\ell \eps + 1) \\
 \ll \frac{\mathcal{M}}{m}   \sum_{n \in \Z} \abs{\widehat{f}(n/m)} + \O(\eps \ell + 1) \ll\frac{\mathcal{M}}{m} \left( \ell X_c + \ell \sum_{t \geq X_{c}} \frac{m^K}{\eps^K n^K \ell^K }\right) + \O(\ell\eps+1) \\ \ll_K \frac{\mathcal{M}}{m} \left( \ell X_c + \frac{\ell m^K}{\eps^K \ell^K X_c^{K-1} }\right) + \O(\ell\eps + 1)\ll \frac{\mathcal{M}}{m} ( \ell + m/\eps) + \O(\ell\eps +1).
\end{multline*}
Finally, setting $\eps:= \sqrt{\mathcal{M}/\ell}$ gives the claimed result. 
\end{proof}

\begin{lemma}\label{characterbd}
Let $A, B, C, D$ be integers, let $m \in \N$, and let $\chi(x) = \left( \frac{x}{m}\right)$ be a Kronecker symbol(a character of modulus $m'$, where $m' = 4m$ if $m \equiv 2 \Mod 4$, and $m = m'$ otherwise). Let $\pi(x):= \chi(Ax + B) \chi(Cx + D)$.
Then: 
$$\mathcal{M} := \max_{n \Mod m'} \Big|\sum_{x \Mod m'}  e(x n/m') \pi(x) \Big| \ll \mathfrak{m},$$ where $\mathfrak{m}$ is given as follows. Let $h = (AD - BC, m)$, and let $\mathcal{P}$ be a set of primes given by 
$$\mathcal{P} :=\{p |m :  2 \nmid \upsilon_p(m),  p \nmid h, \text{ and }(A, C, p) = 1\}.$$ Then:

$$ \mathfrak{m} = \mathfrak{m}_{m, \pi}:= \frac{m}{\prod_{\substack{ p \in \mathcal{P}}} \sqrt{p}/2}.$$

\end{lemma}
\begin{proof}
Since $\left( \frac{\cdot}{2}\right) = \left( \frac{\cdot}{8} \right)$,
\[ \max_{n \Mod 4m} \Big|\sum_{x \Mod 4m}  e\left(\frac{x n}{4m}\right) \pi(x) \Big| \geq  \max_{4n \Mod 4m} \Big|\sum_{x \Mod m'}  e\left(\frac{4 x n}{4m}\right)\pi(x) \Big| >  \max_{n \Mod m} \Big|\sum_{x \Mod m}  \left(\frac{x n}{m}\right)\pi(x) \Big|,\]
so by replacing $m$ with $4m$ when $m \equiv 2 \Mod 4$, it suffices to prove the statement in the case $m = m'$. 

We claim that restricted to such $m$, $\mathcal{M}$ is multiplicative in $m$. Indeed, let $m = \prod_p p^{\alpha_p}$, let $\chi_p:= \left( \frac{\cdot}{p}\right)^{\alpha_p}$ be the Kronecker symbol, and let
\[\pi_p(x):= \chi_p(Ax + B)\chi_p(Cx + D),\] so $\pi(x) = \prod_p \pi_p(x_p)$. Let $y_p$ be the inverse of $\prod_{q \neq p} q^{\alpha_q}$ modulo $p^{\alpha_p}$, and let $y$ be an integer that reduces to $y_p \Mod p^{\alpha_p}$ for all $p$ (so in particular, $(y, m) = 1$). Then for any $x$,

\[x y \sum_p \left(   \prod_{q \neq p} q^{\alpha_q} \right)\equiv x \Mod m,\] 
so, 
\begin{multline*}
\sum_{x \Mod m}  e(x n/m) \pi(x) = \sum_{\substack{x \Mod m}}e\Big(n \Big(y\sum_p x/p^{\alpha_p}\Big)\Big)\prod_p \pi_p(x_p) \\ 
= \hspace{-0.2in} \sum_{\substack{x \Mod m\\ (x_p:= x \Mod p^{\alpha_p})}}  \hspace{-0.2in}e\Big(n \Big(y\sum_p x_p/p^{\alpha_p}\Big)\Big)\prod_p \pi_p(x_p)=  \prod_p \sum_{x_p \Mod p^{\alpha_p}}  e(x_p (ny)/p^{\alpha_p}) \pi_p(x_p).
\end{multline*}
Since we could choose $ny $ to have any set of simultaneous reductions modulo all $p^\alpha|m$, the maximum over $n$ for $\pi$ will be the product of the corresponding maxima for the $\pi_i$'s.

Assume now that $m = p^{\alpha}$ for some prime $p$. If $\alpha$ is even or $p = 2$ or $p |h$, we apply the trivial bound, so assume $\alpha$ and $p$ are odd and $p \nmid h$.

\textbf{Case I:} Suppose $(AC, p) = 1$. Then: 

\begin{align}\label{abc}
\Big|\sum_{x \Mod m}  e(x n/m) \chi(Ax + B) \chi(Cx + D)\Big| &= \Big|\sum_{x \Mod m} e(xn/m) \chi(x + B A^{-1}) \chi(x + D C^{-1})\Big| \nonumber \\
&= \Big|\sum_{x \Mod m} e(xn/m) \chi(x) \chi(x + D C^{-1} - B A^{-1})\Big|   
\end{align}
where the inverses are taken mod $p^{\alpha}$. 
Notice that for any $s$ with $(s, m) = 1$ and $t := s^{-1}  \Mod m$ and for any shift $h$,

\begin{align*}
 \abs{\sum_{x \Mod m} \chi(x) \chi(x + h) e\left(\frac{n x}{m}\right)} &=  \abs{\sum_{x \Mod m} \chi(sx) \chi(sx + sh) e\left(\frac{(nt)(sx)}{m}\right)}\\
&=  \abs{\sum_{x \Mod m} \chi(x) \chi(x + sh) e\left(\frac{(nt)x}{m}\right)}.\end{align*}
Thus, 
$\max_n \abs{\sum_{x \Mod m} \chi(x) \chi(x + h) e\left(\frac{n x}{m}\right)}$ depends only on $(m, h)$, so
\[(\ref*{abc}) = \max_{n \Mod m} \Big|\sum_{x \Mod m} e(xn/m) \chi(x) \chi(x + 1)\Big|.\]
The inner sum depends on the $p$-adic valuation of $n$:
\begin{itemize}
\item When $n \equiv 0 \Mod p^\alpha$, by Lemma \ref*{thetalemma},
$$\abs{\sum_{n \Mod p^\alpha} e(xn /p^\alpha)\chi(x) \chi(x + h)} = \abs{\sum_{n \Mod p^\alpha} \chi(x) \chi(x + 1)} = \abs{\theta(p^\alpha)} = p^{\alpha-1};$$
 
\item When $\upsilon_p(n) = \alpha - 1$, $n =: n' p^{\alpha-1}$, 
\begin{align*}
\sum_{x \Mod p^\alpha} e(n x/p^\alpha)\chi(x) \chi(x + 1) &= p^{\alpha - 1} \sum_{x \mod p} e(x n'/p)\chi(x) \chi(x + 1)\\
& =  p^{\alpha - 1} \sideset{}{^*} \sum_{x \Mod p} e(cx /p)\chi(1 + x^{-1}) \\
&= \pm  p^{\alpha - 1}  \sideset{}{^*}\sum_{x \Mod p} e(x^{-1} /p)\chi(x + 1)  \leq 2 p^{\alpha - 1/2}
\end{align*}
(where $\sum^*$ denotes summation over coprime remainders.)

\item Finally, when $\upsilon_p(n) < \alpha-1$, the sum is $0$ as the Legendre symbol is $p$-periodic for $p \neq 2$. \\
\end{itemize}
In summary, the maximum over $n$ is $2p^{\alpha - 1/2}$ when $p \nmid h$ and $(2AC, p) = 1$.\\

\textbf{Case II:} Next, suppose $(A, p) = 1$ but $p | C$. If $p | D$, the sum vanishes, so without loss of generality, $p \nmid D$, and

\begin{align*}
\Big|\sum_{x \Mod m}  e(x n/m) \chi(Ax + B) \chi(Cx + D)\Big| &= \Big|\sum_{x \Mod m} e(xn/m) \chi(x)\Big|. \\
\end{align*}
\begin{itemize}
\item When $n \equiv 0 \Mod p^\alpha$, $\sum_{x \Mod m} \left( \frac{x}{m}\right) e\left(\frac{n x}{m}\right)= \sum_{x \Mod m} \left( \frac{x}{m}\right) = 0.$
\item When $\upsilon_p(n) = \alpha - 1$, $n =: n' p^{\alpha-1}$, 
$$\abs{\sum_{x \Mod p^\alpha} e(n x/p^\alpha)\chi(x) }= p^{\alpha - 1} \abs{\sum_{x \Mod p} e(x n'/p)\chi(x) }=  p^{\alpha-1}\abs{ \sum_{x \Mod p} e(x/p) \chi(x) }\leq p^{\alpha - 1/2}$$ (Gauss sum).
\item When $\upsilon_p(n) < \alpha-1$, the sum is $0$.
\end{itemize}
In summary, the maximum over $n$ is always at most $p^{\alpha - 1/2}$ when $p$ divides exactly one of $A, C$.\\
We need not consider the case $p | A, C$ since then $p | h$, so this concludes the proof.

\end{proof}

\begin{lemma}\label{twistedsum22}
Let $m, w, D \in \N$ and $q, t \in \Z$ with $(t, D) = 1$ and $wt \equiv q \Mod D$. Let $\chi = \left( \frac{\cdot}{m}\right)$. Then:

$$\sum_{\substack{ N = X \\ N = t \Mod D }}^{X + Y}  \mu^2(N) \chi(N) \chi((wN - q)/D) \ll  \frac{\mathfrak{m} }{mD} Y + \frac{\log X \sqrt{\mathfrak{m}Y} }{\sqrt{D}} + \sqrt{X},$$  
where $$\mathfrak{m} := \frac{m}{\prod_{p \in \mathcal{P}} \sqrt{p}/2}, $$ \[\mathcal{P} :=  \{p |m :  2 \nmid \upsilon_p(m),  p \nmid (q, m) ,  \text{ and }(D, w, p) = 1\}.\] 
\end{lemma}

\begin{proof}
Let $l \in \Z$ be such that $wr = q + l D$. Then: 

\begin{align}\label{eqq}
\sum_{\substack{ N = X \\ N = t \Mod D }}^{X + Y}  \mu^2(N) \chi(N) \chi\left(\frac{wN - q}{D}\right) &=
\sum_{\substack{ N = X \\ N = t \Mod D }}^{X + Y}  \mu^2(N) \chi(N) \chi\left(\frac{w(N-t)}{D} + l\right) \nonumber \\
& =\sum_{\substack{\delta^2 \leq X +Y}} \mu(\delta) \sum_{\substack{s \in [X/\delta^2, (X + Y)/\delta^2]\\ s\delta^2 = t \Mod D }} \chi(s \delta^2) \chi\left(\frac{w(s \delta^2-t)}{D} + l\right) \nonumber\\
&\ll \sum_{\substack{\delta \leq \sqrt{2X}\\(\delta, D) = 1\\(\delta, m) = 1 }} \Big| \sum_{\substack{x \in [\frac{X - t}{D}, \frac{(X + Y)-t}{D}] \\ x = -t D^{-1}  \Mod \delta^2}} \chi(Dx + t) \chi(wx + l) \Big|.
\end{align}

Here we restrict to $(\delta, m) = 1$ since terms with $(\delta, m) > 1$ clearly vanish, and to $(D, \delta) = 1$ because otherwise $s \delta^2 = t \Mod D$ cannot hold (as we assumed $(t, D) = 1$).

For each $\delta \leq \sqrt{2X}$ with $(\delta, m) = (\delta, D) = 1$, pick an integer solution $x_\delta$ to $x_\delta D = -t \Mod \delta^2$, and let \[I_\delta:= \left[\frac{X - t}{\delta^2D} -\frac{x_\delta}{\delta^2}, \frac{X + Y-t}{\delta^2D} - \frac{x_\delta}{\delta^2}\right]\] be an interval of length $Y/(\delta^2 D)$. Changing variables again,
\begin{align*}
(\ref*{eqq}) = &\sum_{\substack{\delta \ll \sqrt{X}\\ (\delta, D) = 1\\(\delta, m) = 1 }} \Big| \sum_{\substack{x \in I_\delta}} \chi(D(\delta^2 x +  x_\delta)  + r) \chi(w( \delta^2x + x_\delta) + l) \Big|\\
\leq& \sum_{\substack{\delta \ll \sqrt{X} \\(\delta, D) = 1\\(\delta, m) = 1 }}  \Big| \sum_{\substack{x \in I_\delta}} \chi(D\delta^2 x + (D x_\delta  + r))\chi(wx \delta^2 + (wx_\delta + l)) \Big|.
\end{align*}
 Note that the determinant \[D \delta^2 \cdot (wx_\delta + l) - (D x_\delta + r) w\delta^2 = \delta^2(D  l - r w) =  -q \delta^2\] satisfies $(q \delta^2, m) = (q, m)$.
 Thus by Lemma \ref*{shortintchar} and Lemma \ref*{characterbd} and using the condition $(\delta, m) = 1$,

\begin{align*}
  \sum_{\substack{x \in I_\delta}} \chi(D\delta^2 x + (D x_\delta  + r))\chi(w \delta^2 x + (wx_\delta + l))  \ll \frac{\sqrt{\mathfrak{m} Y}}{\delta \sqrt{D}} + \frac{\mathfrak{m} Y}{\delta^2 D m} + 1.
\end{align*}
Summing over $\delta$, we conclude that 
\begin{align*} 
\sum_{\substack{ N = X \\ N = t \Mod D }}^{X + Y}  \mu^2(N) \chi(N) \chi\left(\frac{wN - q}{D}\right) &\ll  \frac{\sqrt{\mathfrak{m}Y} \log X}{\sqrt{D}} + \frac{Y \mathfrak{m}}{ D m} + \sqrt{X}.
\end{align*}
\end{proof}

Applying the lemma to $D = d^2, w = r^2$, and $q = 4P$, and using that the terms with $P \nmid N$ contribute an $O(1)$ error term for $P \gg X$, we get an immediate corollary:
\begin{corollary}\label{alltogethercor}
Let $(r, d)$ be an admissible pair of integers and let $n \in \N$. Let $$\mathcal{P}_n =  \{p |n :  2 \nmid \upsilon_p(n), p \neq P\},$$ and let 
$$\mathfrak{n}_n := \frac{n}{\prod_{p \in \mathcal{P}} \sqrt{p}/2}.$$  Then:

\[\mathcal{S}_{d, n, r} \ll \sqrt{X} + \frac{\mathfrak{n}_n }{nd^2} Y + \frac{\log X \sqrt{\mathfrak{n}_nY} }{d}.\]
\end{corollary}

\begin{proof}[Proof of Proposition \ref*{smalldlargem}]
From Corollary \ref*{alltogethercor},

\begin{align*}
\sum_{n \in [Y^\sigma, T]} \sum_{d < Y^\tau} \frac{\mathcal{S}_{d, n, r}}{nd} &\ll \sum_{m \in [Y^\sigma, T]} \sum_{d < Y^\tau} \frac{\sqrt{X}}{nd} + \frac{ \mathfrak{n}_n Y}{n^2 d^3} + \frac{\log X \sqrt{Y} \sqrt{\mathfrak{n}_n} }{nd^2} \\
&\ll \log T \log X \sqrt{X} + Y \hspace{-0.15in}  \sum_{n \in [Y^\sigma, T]}\frac{\mathfrak{n}_n}{n^2} + \sqrt{Y} \log X \sum_{n \in [Y^\sigma, T]}  \frac{\sqrt{\mathfrak{n}_n}}{n}.
\end{align*}
Every integer $n$ is representable uniquely in the form 
\[n = a^2 b \  P^c, \] where $b$ is square-free and $(a, 2P) = (b, 2P) = 1$. In terms of this representation, $\mathcal{P}_n = \{ p |b\}$, so with the divisor bound,

\[\mathfrak{n}_n \leq a^2 b^{1/2 + \eps} \ P^c.\]
Thus
\begin{align*}
\sum_{n \leq T}\frac{\sqrt \mathfrak{n}_n}{n} &\ll \sum_{\substack{c: P^c \leq T}} \sum_{a \leq \sqrt{\frac{T}{P^c}}} \sum_{b \leq \frac{T}{a^2 P^c} }  \frac{1}{P^{c/2} \ a\  b^{3/4 - \eps}} \ll   \sum_{a \leq \sqrt{T}}(1/a) \left( T/a^2\right)^{1/4 + \eps} \ll T^{1/4 + \eps}.\\
\end{align*}
Similarly, 
$$\sum_{n > Y^\sigma} \frac{\mathfrak{n}_n}{n^2} \ll\sum_{c} \sum_{a \in \N} \frac{1}{ P^c a^2} \sum_{b > Y^\sigma/a^2P^c} b^{-3/2 + \eps} \ll\sum_{a \leq Y^{\sigma/2}} \frac{1}{a^2}\frac{a}{Y^{\sigma/2 - \eps}} +   \sum_{a > Y^{\sigma/2}} \frac{1}{a^2} \ll Y^{-\sigma/2 + \eps}. $$
To summarize,
$$\sum_{n \in [Y^\sigma, T]} \sum_{d < Y^\tau} \frac{\mathcal{S}_{d, n, r}}{nd}  \ll \log T \log X \sqrt{X} + Y^{1 - \sigma/2 + \eps} + \sqrt{Y} T^{1/4 + \eps} \log X $$ as desired.
\end{proof}
\subsubsection{Proof of Proposition \ref*{hardaveragemid}}\label{propproof}
\begin{proof}
Plugging Proposition \ref*{smalldn} and \ref*{smalldlargem} into identity (\ref*{identity}) with $\tau = 1/18$, $\sigma = 5/9$ gives 
\begin{multline*}
\sum_{\substack{ N \in [X, X + Y] \\ N \sqf \\ P \nmid N }} H_1(r^2 N^2 - 4 PN) =   \frac{Y B \nu(r)}{\zeta(2)}  \frac{ \sqrt{4 P X - r^2X^2}}{\pi }  +\\
+\O\Big(\Big( \frac{YPX}{T}+ \frac{Y^2 P^{\frac{1}{2}}}{ X^{\frac{1}{2}}} + r Y^{\frac{3}{2}} X^{\frac{1}{2}} +  (PX)^\frac{1}{2}Y^{\frac{8}{9}} 
+\sqrt{P} X Y^{\frac{5}{18}}  + \sqrt{Y} T^{\frac{1}{4}} \sqrt{PX} \Big)(PXT)^\eps\Big) ,
\end{multline*}
Setting \[T:= (YPX)^{\frac{2}{5}}\] gives the claimed error term. 
\end{proof}

\subsubsection{Proof of Proposition \ref*{hardaverage}}\label{propdeduction}
It remains to bound the contribution of $r$ with $2 \sqrt{\frac{P}{X + Y}}\leq r \leq 2 \sqrt{\frac{P}{X}}$. Firstly, note that for $r$ in this range and $N \in[X, X + Y]$,
\[4 P N - r^2 N^2 \leq 4PN\left(1 - \frac{N}{X + Y} \right) \ll YP,\] so $H_1(r^2 N^2 - 4 PN) \ll (YP)^{1/2 + \eps}$ and

\begin{align*}
\frac{1}{Y X} \sum_{\substack{ N \in [X, X + Y] \\ P \nmid N }} \sum_{r = 2 \sqrt{\frac{P}{X + Y}}}^{ 2 \sqrt{\frac{P}{N}}} H_1(r^2 N^2 - 4 PN) &\ll  \frac{1}{XY} \cdot Y \sqrt{P}\left(\frac{1}{\sqrt{X}} - \frac{1}{\sqrt{X + Y}}\right)  (PY)^{1/2 + \eps}  \\
&\ll P^{1 + \eps} Y^{3/2}  X^{-5/2 + \eps}.
\end{align*}
Moreover, since $\nu(r) \ll 1$,
\[\sum_{2 \sqrt{\frac{P}{X + Y}}}^{ 2 \sqrt{\frac{P}{X}}} \nu(r)  \sqrt{4 y- r^2} \ll \frac{\sqrt{PY}}{X} \cdot  \sqrt{P}\left(\frac{1}{\sqrt{X}} - \frac{1}{\sqrt{X + Y}}\right) \ll  PY^{3/2} X^{-5/2}. \]  
 Since $PY^{3/2} X^{-5/2} \ll  PY^{1/2} X^{-3/2}$, this with Proposition \ref*{hardaveragemid} concludes the proof by summing the error term in the Proposition over $r \ll \sqrt{P/X}$.

\subsection{$P$-divisible levels, remaining terms, and the dimension formula}\label{allotherpn}

In this section we conclude Theorem \ref*{mainthm} from Propositions \ref*{easyaverage}, \ref*{hardaverage}. First, we address the levels $N$ divisible by $P$, which have been excluded from consideration so far. When $P | N$, $\lambda_f(P) \sqrt{P} = \pm 1$, so we trivially get that 
\[\frac{1}{XY} \sum_{\substack{N \in [X, X + Y] \\ P | N}} \abs{\sqrt{P} \lambda_f(P) \eps(P)} \ll \frac{1}{XY} \cdot k X \left(\frac{Y}{P} + 1\right) \ll \frac{k}{P} + \frac{k}{Y} .\]
To sum the $\delta_{k = 2}P$ term, we use the well-known asymptotic
\begin{align}\label{squarefree}
\sum_{N \leq Z} \mu^2(N) = \frac{Z }{\zeta(2)} + \O(\sqrt{Z}).
\end{align}
Combining Proposition \ref*{easyaverage}, Proposition \ref*{hardaverage}, and (\ref*{squarefree}), expressing the error term in terms of $\delta_2$ and $\delta_1$, and using that $U_{k - 2}(x)\ll k\text{ for }x \in [-1, 1],$ we conclude that

\begin{align}\label{staar}
\frac{\zeta(2) \pi}{Y X }\sideset{}{^\square}\sum_{ N \in [X, X + Y]}\sum_{f \in H^{\text{new}}(N, k)}  \lambda_f(P) \sqrt{P} \  \eps(f)  &= A \sqrt{y}  +  (-1)^{k/2 - 1} B \sum_{1 \leq r \leq 2 \sqrt{y}} \nu(r) \sqrt{4 y - r^2} U_{k - 2}\left(\frac{r}{2 \sqrt{y}} \right) \nonumber \\
&  - \delta_{k = 2}\pi  y + \O_\eps\left( kX^{-\delta' + \eps } +  \frac{k}{Y}+ \frac{k}{P} \right):= (*)
\end{align}
where 
\[\delta' = \min \left\{\frac{5}{19} - \frac{11 \delta_1}{19} - \frac{16 \delta_2}{19}, \frac{1}{5} - \frac{11 \delta_1}{10} - \frac{2 \delta_2}{5}, \frac{1}{9} - \delta_1 - \frac{ \delta_2}{9}, \frac{2}{9} - \delta_1 - \frac{13 \delta_2}{18}, \frac{\delta_2}{2} - \delta_1. \right\}\]
Solving this for $0 < \delta_{1, 2} < 1$, this minimum is positive exactly when
\[0 < \delta_1 < 1/11, \ \  2 \delta_1 < \delta_2 < 1/13 (4 - 18 \delta_1).\]

Finally, we address the denominator in Theorem \ref*{mainthm}.

 From the work of Martin (\cite{greg}), for a square-free level $N$, the dimension of the space of cusp newforms for $\Gamma_0(N)$ is given by
\begin{align}\label{dimformula}
\dim S^{\text{new }}(N, k) = \frac{(k-1) \phi(N)}{12} + O(N^\eps).
\end{align}
We can compute the size of this expression on average:
\begin{lemma}\label{phisum}

One has 

$$\sum_{n \leq Z} \mu^2(n) \phi(n) =  \frac{Z^2}{2\zeta(2)} \prod_p\left(1 - \frac{1}{p^2 + p} \right) +  \O(Z^{3/2 + \eps})$$
\end{lemma}

\begin{proof}
We use the asymptotic 
\[\sum_{m \leq X/a} \phi(am) = \frac{1}{2 \zeta(2)} \frac{1}{a \prod_{p | a} (1 + 1/p)} X^2 + O(X \log X)\]
(see \cite{Postnikov}, section $4.2$).  
From this, 
\begin{multline*}
 \sum_{n \leq Z} \mu^2(n) \phi(n)=\sum_{d \leq \sqrt{Z}} \mu(d) \sum_{m \leq Z/d^2} \phi(d^2 n)   = \frac{Z^2}{2 \zeta(2)} \sum_{d \leq \sqrt{Z}}\frac{ \mu(d)}{d^2 \prod_{p | d} (1 + 1/p)} + \O\left( Z^{3/2 + \eps} \right)\\
= \frac{Z^2}{2\zeta(2)} \prod_p\left(1 - \frac{1}{p^2 + p} \right) +  \O\left( Z^{3/2 + \eps} + Z^{2 + \eps}/\sqrt{Z} \right)
\end{multline*}
as aimed.
\end{proof}

From Lemma \ref*{phisum},

\begin{align}\label{dimdim}
 \sideset{}{^\square}\sum_{\substack{ N \in [X, X + Y] }} \sum_{f \in H^{\text{new}}(N, k)} 1 &=  \sum_{\substack{ N \in [X, X + Y] }}  \frac{k-1}{12} \phi(N) \mu^2(N) + \O( Y X^\eps ) \nonumber\\
&= \frac{ (k-1) (X Y)}{12 \cdot \zeta(2)} \prod_p\left(1 - \frac{1}{p^2 + p} \right)  +\O(kX^{3/2 + \eps}),
\end{align}
and 
\[ \frac{1}{ \sum^\square_{\substack{ N \in [X, X + Y] }} \sum_{f \in H^{\text{new}}(N, k)} 1} = \frac{12 \cdot \zeta(2) \prod_p\left(1 - \frac{1}{p^2 + p} \right)^{-1}}{ (k-1)( X Y )}  +  \O\left(\frac{1}{k Y^2 X^{1/2 - \eps}}\right).\]  

Now, $(\ref*{staar})$ implies that the the numerator of the left-hand side of Theorem \ref*{mainthm} is bounded by $X Y k y + XY k = k Y P + k Y X.$ Thus 

\begin{align*}
\frac{\sum^\square_{N \in [X, X + Y]}\sum_{f \in H^{\text{new}}(N, k)}  \lambda_f(P) \sqrt{P}   \eps(f) }{ \sum^\square_{ N \in [X, X + Y] } \sum_{f \in H^{\text{new}}(N, k)} 1} =  (*)\cdot  \frac{  12 \prod_p\left(1 - \frac{1}{p^2 + p} \right)^{-1} }{(k - 1) \pi} + \O\left(\frac{P}{X^{1/2-\eps} Y}  + \frac{X^{1/2 + \eps}}{Y}\right).
\end{align*}
This is again a power saving error term in the range of $\delta_{1, 2}$ as above. Simplifying the Euler products then completes the proof of Theorem \ref*{mainthm}.

\section{Geometric Averaging}
In this section we complete the proof of Theorem \ref*{thm2} and analyze the asymptotic behavior of the dyadic average. 
\begin{proof}[Proof of Theorem \ref*{thm2}]
Let $Z:= cX$, and let $\delta_2$ be a parameter chosen depending on $\delta_1$ to satisfy the conditions of Theorem \ref*{mainthm} with a powersaving error term. Assume further that $Y \sim X^{1 - \delta_2}$ is chosen so $Y$ divides $Z - X$; let $X = X_1, X_2, \ldots, X_G = Z - Y$ be given by $X_g = X + (g-1)Y,$ where $G:= (Z-X)/Y$. From (\ref*{staar}) and from (\ref*{dimdim}),
\begin{align*}
\sideset{}{^\square}\sum_{\substack{ N \in [X, Z] }}\sum_{f \in H^{\text{new}}(N, k)}  \lambda_f(P) \sqrt{P}  \eps(f) &= \sum_g \sideset{}{^\square}\sum_{\substack{ N \in [X_g, X_{g + 1}] }}\sum_{f \in H^{\text{new}}(N, k))}  \lambda_f(P) \sqrt{P}  \eps(f)\\
&= \sum_g  \frac{(k - 1) \prod_p\left(1 - \frac{1}{p^2 + p} \right) }{ 12 \cdot \zeta(2)}  X_g Y  \mathcal{M}_k\left(\frac{P}{X_g}\right) + o\left(X^{2 - \delta}\right)
\end{align*}  
for some small $\delta >0$.
Since for $u \in [X, Z]$, $\frac{P}{u} - \frac{P}{u + Y} \ll \frac{PY}{X^2} = o(1),$ we can approximate this sum as $X \to \infty$ by the integral, so

\begin{align*}
\sum_g X_g Y \mathcal{M}_k(P/X_g) =  \int_X^Z u \mathcal{M}_k(P/u) du + o(1) =  X^2 \int_1^c \mathcal{M}_k(y/u) u du.
\end{align*}
Finally, from (\ref*{dimformula}), we can again compute that 

\[ \frac{1}{\sum^\square_{\substack{ N \in [X, Z] }} \sum_{f \in H^{\text{new}}(N, k)} 1} = \frac{24 \cdot \zeta(2) \prod_p\left(1 - \frac{1}{p^2 + p} \right)^{-1}}{(c^2 - 1)(k - 1) X^2}  +  o\left(X^{-2}\right),\]  
so

\begin{align*}
\frac{\sum^\square_{\substack{ N \in [X, Z]}}\sum_{f \in H^{\text{new}}(N, k))}  \lambda_f(P) \sqrt{P}  \eps(f) }{\sum^\square_{ N \in [X, Z]}\sum_{f \in H^{\text{new}}(N, k))} 1} =   \frac{2}{ (c^2 - 1)} \int_1^c \mathcal{M}_k(y/u) u du  + o_c(1)
\end{align*} 
as aimed. 

Finally, when $k = c = 2$ and $y \in [0, 1]$, the integral becomes
\[\frac{2}{ 3} \int_1^2  \alpha \sum_{1 \leq r \leq 2 \sqrt{y/u}} \nu(r) \sqrt{4 y u - r^2 u^2} + \beta \sqrt{y u} - \gamma y  \ du.\]
For $y < 1/4$, the sum on $r$ is empty, so the integral evaluates to
\[\frac{4}{9}(2^{3/2} - 1) \beta \sqrt{y} - \frac{2}{3} \gamma y.\]
On $[1/4, 1/2]$, there is an additional term $r = 1$ that appears for $u < 4y$, that is,
\begin{multline*}
\frac{2\alpha}{3} \int_1^{4y} \sqrt{4y u - u^2} du = \frac{2\alpha}{3}  4 y^2 \int_{1/2y - 1}^1 \sqrt{1 - v^2} dv \\
=  \frac{2\alpha}{3} y^2 \pi + \frac{2\alpha}{3}2 y^2  \arcsin(1 - 1/2y) + \frac{2\alpha}{3}(2y - 1)\sqrt{y - 1/4}. 
\end{multline*}
Finally, for $1> y > 1/4$, the $r = 1$ term is present for all $u$, and we get 
\begin{multline*}
\frac{2\alpha}{3} \int_1^{2} \sqrt{4y u - u^2} du = \frac{2\alpha}{3}  4 y^2 \int_{1/2y - 1}^{2/2y - 1} \sqrt{1 - v^2} dv \\
=  \frac{2\alpha}{3}2 y^2  \arcsin( 1/y-1 ) + \frac{2\alpha}{3}2 (1- y )\sqrt{2y - 1} + \frac{2\alpha}{3}2 y^2  \arcsin(1 - 1/2y) + \frac{2\alpha}{3}(2y - 1)\sqrt{y - 1/4}. 
\end{multline*}
\end{proof}

\section{Properties of the density function $\mathcal{M}_k$.}
In this section we prove Theorem \ref{tm3}. For $t > 0$, let 
\[f_c(x) := \sqrt{1 - t^2 x^2} U_{k - 2} (tx) \chi_{[-1/t, 1/t]},\] where $\chi$ denotes the characteristic function of an interval. Then
\[\widehat{f}_t(s) = (-1)^{k/2 - 1} \frac{( k- 1) J_{k - 1} ( 2\pi \abs{s}/t)}{2 \abs{s}}.\] 
Recall that 
\[Q(d)= \mu^2(d) \prod_{p | d} \frac{p^2}{p^4 - 2 p^2 - p + 1}\] and 
\[ \nu(r) =  \prod_{p | r}\left( 1 + \frac{p^2}{p^4 - 2 p^2 - p + 1}\right) = \sum_{\substack{d | r \\ d \in \N}} Q(d),\]
so
\begin{align*}
\sum_{1 \leq r \leq 2 \sqrt{y}} \nu(r) \sqrt{4 y - r^2} U_{k - 2}\left(\frac{r}{2 \sqrt{y}} \right)  &= \sum_{\substack{ d \in \N}} Q(d) 2 \sqrt{y} \sum_{1 \leq t \leq \frac{2 \sqrt{y}}{d}} \sqrt{1 - \frac{d^2 t^2}{4 y}}  U_{k - 2}\left(\frac{dt }{2 \sqrt{y}} \right) \\
  &= \sum_{\substack{ d \in \N}} Q(d)  \sqrt{y} \sum_{s \in \Z \setminus \{0\}} f_{\frac{d}{2 \sqrt{y}}} (s).
\end{align*}
Applying Poisson summation to $f_{d/2 \sqrt{y}}$, 
\[ \sum_{s \in \Z \setminus \{0\}} f_{\frac{d}{2 \sqrt{y}}} (s)= \sum_{s \in \Z} \widehat{f}_{\frac{d}{2 \sqrt{y}}} (s) -  U_{k - 2}(0)=  \sum_{s \in \Z} \widehat{f}_{\frac{d}{2 \sqrt{y}}} (s) +  (-1)^{k/2},\]  that is, 
\begin{align*}
\sum_{1 \leq r \leq 2 \sqrt{y}} \nu(r) \sqrt{4 y - r^2} U_{k - 2}\left(\frac{r}{2 \sqrt{y}} \right)  &= \sqrt{y} \sum_{\substack{d \in \N}} Q(d)\sum_{t \in \Z} \widehat{f}_{\frac{d}{2 \sqrt{y}}} (t) + \sqrt{y} \sum_{\substack{  d \in \N}} Q(d)  (-1)^{k/2}.
\end{align*}
From the matching of Euler product factors at every $p$ that
\[\sum_{d \in \N} Q(d) = \beta/\alpha,\] and thus we can rewrite the function $\mathcal{M}_k(y)$ as
\begin{align*}
(k - 1) \mathcal{M}_k(y)  &=   \sqrt{y} (-1)^{\frac{k}{2} - 1} \alpha  \sum_{  d \in \N} Q(d)  \sum_{t \in \Z} \widehat{f}_{\frac{d}{2 \sqrt{y}}} (t)   - \gamma \delta_{k = 2}  y   \\
&=   \sqrt{y} \alpha  \sum_{ d \in \N} Q(d)  \sum_{s \in \Z} \frac{(k - 1) J_{k - 1}(4 \pi |s|\sqrt{y}/d)}{2 |s|}  -\gamma  \delta_{k = 2}  y   \\
&=     \sqrt{y} \alpha  \sum_{ d \in \N} Q(d)  \sum_{s \in \N} \frac{(k - 1)J_{k - 1}(4 \pi |s|\sqrt{y}/d)}{ |s|} +  \delta_{k = 2}   \pi y \alpha    \sum_{  d \in \N} Q(d)/d  - \gamma \delta_{k = 2}  y .
\end{align*}
Observe again from the matching of the $p$-part of the Euler products that

\begin{multline*}
(\alpha/\gamma) \sum_{d \in \N}  Q(d)/d = \frac{\pi}{6} \prod_p\left(\frac{1 - p - 2 p^2 + p^4}{p^4 - 2 p^2 + p}\right) \left( \frac{-1 + p + p^2}{p (1 + p)} \right) \left( 1 + \frac{p^2}{p^4 - 2 p^2 - p + 1}\right)\\
=  \frac{\pi}{6} \prod_p\left( 1 - \frac{1}{x^2} \right) = \frac{1}{\pi}.
\end{multline*} Thus the $\delta_{k = 2}$ terms cancel out, and we arrive to the formula in Theorem \ref*{tm3}: 
\[\mathcal{M}_k(y) =    \sqrt{y} \alpha \sum_{d \in \N} Q(d)  \sum_{s \in \N} \frac{J_{k - 1}(4 \pi s\sqrt{y}/d)}{s} .\]
Finally, applying the asymptotic \[J_{k - 1}(z) = (-1)^{k/2} \sqrt{\frac{2}{\pi z}} \cos(z - 3 \pi/4) + \O(\min\{1/z^{3/2}, 1\}),\] we get the asymptotic expansion
\[\mathcal{M}_k(y) \sim  (-1)^{k/2} y^{1/4}  \sqrt{\frac{2}{\pi}}  \alpha \sum_{d \in \N} \frac{Q(d) \sqrt{d}}{s^{3/2}} \sum_{s \in \N}\cos \left(\frac{4 \pi s\sqrt{y}}{d} - \frac{3 \pi}{4} \right)\] with an error term of \[ \sum_{s \in \N}1/s \left(y^{-1/4}\sum_{d \ll s \sqrt{y}} d^{-1 /2} s^{-3/2} + \sqrt{y}\sum_{d \gg s \sqrt{y}} 1/d^2\right) \ll 1,\] as aimed.

\section{Asymptotics of Smooth Geometric Averages}

Finally, we analyze the asymptotic behavior of the smoothed murmuration function.

\begin{proof}[Proof of Theorem \ref*{thm4}]
Let $f(x) =1/x J_{k - 1} (1/x)$. The Mellin transform of $f$ is given by 
\begin{align*}
\tilde{f}(s)  = \frac{1}{2^{s}} \frac{\Gamma\left(\frac{k}{2}  -  \frac{s}{2}\right)}{\Gamma\left(\frac{k}{2}  + \frac{s}{2}\right)}
\end{align*}
on the strip $(k - 1, k )$.
By Stirling approximation, on this strip,

\[\abs{\tilde{f}( \sigma + i t)} \ll t^{- \Re(s)}, \] so we can apply Mellin inversion, which implies that
\[F(x) :=  \sum_{d \geq 1}d Q(d) f(d/x) =  \sum_{d \geq 1}Q(d) \frac{1}{2 \pi i} \int_{\Re(s) = k -1/2} \tilde{f}(s) x^s d^{-s+1}  ds  = \frac{1}{2 \pi i}\int_{\Re(s) = k - 1/2} L(s-1) \tilde{f}(s) x^s ds,\] where 

\begin{multline*}
L(s):= \sum_d Q(d) d^{-s} = \prod_p   \left(1 +  \frac{p^2}{p^4 - 2 p^2 - p + 1} p^{-s}\right) \\
=  \frac{ \zeta(s + 2)}{\zeta(2s + 4)} \prod_p \left( 1 + \frac{-1 + p + 2 p^2}{(1 - p - 2 p^2 + p^4) (1 + p^{2 + s})}\right). 
\end{multline*} For a function $\Phi:(0, \infty) \to \R$ of compact support, let  
\[F_\Phi(x):= \left(\int_0^\infty F\left( \frac{x}{u}\right) \Phi(u) u^2\frac{d u}{u} \right)/\tilde{\Phi}(2).\]
Then by Mellin inversion, 
\begin{align}\label{smoothedin}
F_\Phi(x) = \frac{1}{2 \pi i}\int_{\Re(s) = k - 1/2} L(s-1) \tilde{f}(s) \tilde{\Phi}(s + 2) x^s ds.
\end{align}
The Euler product part in the expression above converges uniformly for $s > \sigma$ for any $\sigma > -3$, so evaluated at $s-1$, it is analytic and uniformly bounded in the region $\Im z > -\sigma$ for $\sigma > -2$.
The Mellin transform $\tilde{\Phi}$ is entire because of the support assumption, and if $\Phi$ is smooth, the decay of $\Phi$ in the $t$ aspect allows us to shift the contour in (\ref*{smoothedin}) to the line $\Re(s) = - \sigma$ for some $\sigma > -1/2$ (indeed, the $L$-function grows at most like $\abs{t}^{\sigma/2}$ on this strip, whereas $\tilde{f} \ll \abs{t}^{\sigma}$). The residue at $s = 0$ is given by 

\[\mathrm{r} := \frac{1}{\zeta(2)}  \prod_p \left( 1 + \frac{-1 + p + 2 p^2}{(1 - p - 2 p^2 + p^4) (1 + p)}\right) \widetilde{\Phi}(2). \] In summary, 

\[F_\Phi(x) = r + O_\sigma(x^{-\sigma}), \sigma \in [-1/2, 0]. \]
We now apply this to our function. The function \[\mathcal{M}^k_\Phi(x) = \int_0^\infty \left(\mathcal{M}_k(x/u) \Phi(u) u^2 \frac{du}{u}\right)/\left(\int_0^\infty\Phi(u) u^2 \frac{du}{u}\right)\] converges to $1/2$ if an only if $\mathcal{M}^k_\Phi(x^2)$ does. Furthermore,

\[\int_0^\infty \mathcal{M}_k(x^2/u) \Phi(u) u^2 \frac{du}{u} =  \int_0^\infty \mathcal{M}_k\left(\frac{x^2}{v^2}\right)\left( 2 v^2 \Phi\left(v^2\right) \right)   v^2 \frac{dv}{v} =  \int_0^\infty \mathcal{M}_k\left(\frac{x^2}{v^2}\right)\Phi_1(v)   v^2 \frac{dv}{v}, \] where 
\[\Phi_1(v) =  2 v^2 \Phi\left(v^2\right),\] and 
\[\widetilde{\Phi_1}(2) = 2\int_0^\infty v^2 \Phi(v^2) v dv=  \int_0^\infty v\Phi(v)  dv = \tilde{\Phi}(2).\]
Thus it suffices to prove the statement for the function $\mathcal{M}'_k(y) = \mathcal{M}_k(y^2)$.  Now, setting $x_s = 4 \pi s y$ for $s \in \N$, we have
\begin{multline*}
 \mathcal{M}'_k(y/u) =\frac{\alpha}{4 \pi}  \sum_{d, s \in \N}   \frac{4 \pi s y}{u }  Q(d)  \frac{J_{k - 1}\left(\frac{4 \pi sy/u}{d}\right)}{s^2}
\\= \frac{\alpha}{4 \pi} \sum_{s \in \N} \frac{x_s}{u} \sum_d Q(d) \frac{J_{k - 1}\left(\frac{x_s}{d u}\right)}{s^2} = \frac{\alpha}{4 \pi} \sum_{s \in \N} \frac{1}{s^2} F(x_s/u)\end{multline*}
Hence, 
\begin{multline*}
\mathcal{M}'_\Phi(y) =  \frac{\alpha}{4 \pi} \sum_{s \in \N} \frac{1}{s^2} F_\Phi(x_s) /\tilde{\Phi}(2)=   \frac{\alpha}{4 \pi} \zeta(2) \frac{r}{\tilde{\Phi}(2)} + O\left(\sum_{s \in \N} \frac{1}{s^2} (sy)^{-\sigma}\right)\\ =  \frac{\alpha }{4 \pi} \prod_p \left( 1 + \frac{-1 + p + 2 p^2}{(1 - p - 2 p^2 + p^4) (1 + p)}\right) + \O(y^{-\sigma}) = \frac{1}{2} + O(y^{-\sigma}).
\end{multline*}
\end{proof}

\section{Asymptotics of Sharp Geometric Averages}

In this section we prove Theorem \ref*{thm5}.

\begin{lemma}\label{J}
For every $n \in \N$ and every $K \geq n + 1$ such that $2 \nmid n + K$, the indefinite integral 
\[\int J_{K}(x)/x^n dx \] can be expressed as 
\[ \sum_{t \geq n } \frac{c_t J_t(x)}{x^{n}}, \] where $c_t \in \R$ are coefficients depending on $K, n$.
\end{lemma}
\begin{proof}
We prove this by simultanious induction on $K-n$. Firstly, for all $K$, 
\[\int J_{K}(x)/x^{K-1} dx = - J_{K- 1}(x)/x^{K-1}.\] Suppose now that the statement holds for all $n$ for indices up to $K:= n + 1+ 2T$. Then:
\[\int J_{K+2}(x)/x^n dx = 2(K + 1) \int J_{K + 1}(x)/x^{n + 1} dx - \int J_{K}/x^n dx.\] Applying the induction step and using the relation 
\[2k J_k(x)/x = J_{k - 1} (x) + J_{k + 1}(x)\] then completes the proof.
\end{proof}
Applyinf asymptotics of Bessel functions at infinity, we get an immediate corollary:
\begin{corollary}\label{integralbound}
For all $K \geq 5$, is a choice of an indefinite integral such that
\[\int J_{K}(u)/u^4 du \ll u^{-4.5}.\]
\end{corollary}

\begin{lemma}\label{Qcount}
Assume RH for $\zeta(s)$. For a parameter $T$,
\[\sum_{d \leq T} Q(d) = L - \frac{\Delta}{T} + R(T),\] where $L = \beta/\alpha$, $R(T) \ll T^{-2 + \eps}$, and \[\Delta:= \frac{1}{\zeta(2)} \prod_p \left(1 +  \frac{2 p-1 }{(x^4 - 2 x^2 - x + 1)}\right).\]
\end{lemma}
\begin{proof}
Let 
\begin{multline*}
L(s):= \sum_d Q(d) d^{-s} = \prod_p   \left(1 +  \frac{p^2}{p^4 - 2 p^2 - p + 1} p^{-s}\right) \\
=  \frac{ \zeta(s + 2)}{\zeta(2s + 4)} \prod_p \left( 1 + \frac{-1 + p + 2 p^2}{(1 - p - 2 p^2 + p^4) (1 + p^{2 + s})}\right). 
\end{multline*} 

As this series converges absolutely for $\sigma > 0$, we have by Perron's formula that for $T$, 
\[\sum_{d < T} Q(d) = \frac{1}{2 \pi i} \int_{1 \pm i \infty}L(s) T^s \frac{ds}{s}.\] 

%Using the classical zero-free region of $\zeta$ \[1 - \frac{c}{\log(2 + \abs{t})},\] we 
Shifting the contour to the line $-2 + \eps$, we pick up a pole at $0$ with residue $L$, and at $-1$ with residue $-\Delta$, with the claimed error term. 

\end{proof}
We now prove Theorem \ref*{thm5}.
\begin{proof}[Proof of Theorem \ref*{thm5}]
For convenience of notaiton, we analyze $\mathcal{M}^k_c( y^2/(4 \pi)^2)$ to show onvergence to $1/2$.
From Lemma \ref*{J}, we can express 
\[\int J_{k - 1}(x)/x^4 dx =: J(x)/x^4 = \sum_{t \geq n} c(t) J_t(x)/x^4\] for some linear combination $J$ of Bessel functions of integer index $\geq 4$ with $c(t) \in \R$ and $c(t) = 0$ for all but finitely many $t$'s. Note that 
\[\int_0^\infty J(x)/x dx = -1/4.\] Then 
\begin{align}\label{JJ}
\int_1^c  \mathcal{M}_k( y^2/(4 \pi)^2 / u) u du &= \frac{\alpha}{4 \pi} \sum_{d, s} \frac{Q(d)}{s} y \int_1^c \sqrt{u} J_{k - 1}\left(\frac{s y}{d\sqrt{u}}\right) du \nonumber \\
 &=  \frac{\alpha}{2 \pi}  \sum_{d, s} \frac{ Q(d)}{s} y \int_1^{\sqrt{c}} u^2 J_{k - 1}\left(\frac{ys }{du }\right) du \nonumber \\
 &=  \frac{\alpha}{2 \pi}  \sum_{d, s} \frac{ Q(d)}{s} y \int_{1/\sqrt{c}}^1 J_{k - 1}\left(\frac{ys u }{d }\right)/u^4 du \nonumber \\
 &=  \frac{\alpha}{2 \pi}  \sum_{d, s} \frac{ Q(d)}{d^3} s^2 y^4 \int_{1/\sqrt{c} (ys/d)}^{ys/d} J_{k - 1}\left(u\right)/u^4 du\nonumber  \\
 &=  \frac{\alpha}{2 \pi}  \sum_{d, s} \frac{d Q(d)}{s^2}  \left(J\left(\frac{y s}{d \sqrt{c}} \right)c^2 - J\left(\frac{ys}{d}\right) \right)
\end{align} 
(which converges since $\frac{1}{s^2}\sum_{d > ys} d Q(d) J(ys/d) \ll \frac{1}{s^2} ys \sum_{d > ys} \frac{1}{d^2} \ll 1/s^2$.)
We analyze the function
\[\sum_{d \in \N}d  Q(d) J(a/d)\] for $a = a_s(u) = ys /uc'$ where $c' \in \{1, \sqrt{c}\}$. We need to understand this function up to a $o(1)$ error term, as by convergence such an error term would contribute $o(1)$ to (\ref*{JJ}). 
Let \[\phi(x) := x J(a/x).\] Then 
\[\phi'(x) = J(a/x) - (a/x) J'(a/x) = J(a/x) - J_{k - 1}(a/x) - 4 J(a x)/x.\]Since $J$ does not involve Bessel functions of index less than $2$, $\phi(x) \to 0$ as $x \to \infty$. Hence from Abel summation

\begin{align}\label{mainterm}
&\sum_{d} Q(d) a J_{k - 1}(a/d) = -\int_1^\infty (\beta - \frac{\Delta}{x} + R(x))\phi'(x) dx \nonumber \\
&=  -\int_1^\infty L \phi'(x) dx +  \int_1^\infty  \frac{\Delta}{x}\phi'(x) dx \nonumber  -\int_1^\infty R(x)\phi'(x) dx \nonumber \\ 
&:= -I + II - III.
\end{align}
We analyze this term by term.
\subsection{I}
Note that
\[I = -L x J(a/x) |_1^\infty= \beta J(a).\] This decays like $y^{-1/2}$ as $y \to \infty$.

\subsection{II}

\begin{align*}
II &= \int_1^\infty \frac{\Delta}{x} \phi'(a/x) dx =\frac{\Delta}{x}x J(a/x) |_1^\infty - \Delta \int_1^\infty \left(\frac{1}{x}\right)' x J (a/x) dx  \\
&\Delta J(a) + \Delta  \int_0^a  J( x) \frac{dx}{x}  = \Delta J(a) + \Delta(1/4) - \Delta \int_a^\infty  J( x)/xdx  .
\end{align*} Once again, the first and third terms decay as $y \to \infty$, and the second term contributes $\Delta/4$.

\subsection{III}

From the definition of $R(T)$ in Lemma \ref*{Qcount}, we can find a function $E: \N \to \R$ with $E(n) \ll n^\delta$ for any $\delta > 0$ such that for $t, T(t) := \floor{t}$, 
\[R(t) = \frac{E(T)}{T^2} - \frac{\Delta}{t} + \frac{\Delta}{T} =  \frac{E(T)}{T^2}  + \frac{\{t\}}{T t}\] so
\[III =  \int_1^\infty \frac{E(T)}{T^2}\phi'(x) dx +  \Delta \int_1^\infty  \frac{\{t\}}{T t}\phi'(x) dx =: a + b.\]
Then: 
\[a = \sum_{n \in \N} \frac{E(n)}{n^2} \int_{n}^{n+1} \phi'( x) dx = \sum_n \frac{E(n)}{n^2} \left(n J(a/n) - (n + 1) J(a/(n + 1)\right) =: \sum a_n.\]
Note that 
\[x J(a/x) \ll x  \sqrt{x}/\sqrt{a},\] so 
\[\sum_{n < a^p} a_n \ll a^{p/2 + \delta- 1/2} = o(1)\] as long as $p < 1$. 
Next, note that 
\[x J(a/x) = a f(a/x),\] where  $f(x) = J(x)/x$ is a function with a uniformly bounded derivative. Hence 
\[ \sum_{n > a^p} a_n =  \sum_{n > a^p} \frac{1}{n^{2 - \delta}} a \left(\frac{a}{n} -\frac{ a}{n + 1}\right) = \frac{a^2}{n^{3p-\delta}} = o(1)\] as long as $p > 2/3$.

Finally we address $b$. Note that 
\begin{align*}
\frac{1}{T}\int_{T}^{T + 1} \frac{\{t\}}{t }\phi'(t) dt &=\frac{J(a/(T + 1))}{T} - \frac{1}{T} \int_T^{T + 1} \left(\frac{t - T}{t }\right)' t J(a/t) dt\\
& =\frac{J(a/(T + 1))}{T} -  \int_{a/T}^{a/(T + 1)}  J_{k - 1}(t)/t \ dt \\
\end{align*} so 

\begin{align*}
b &=-  \Delta \int_0^a J(x)/x dx + \Delta \sum_{n \in \N} \frac{J(a/(n + 1))}{n}\\
& = - \Delta (1/4)+  \Delta \int_a^\infty J(x)/x dx + \Delta \sum_{n \in \N} \frac{J(a/(n + 1))}{n}
\end{align*}
where the second term is again $o(1)$ since the integral from $0$ to $\infty$ converges. It remains to estimate the sum in the third term.
From the asymptotics of Bessel Functions, for a large parameter $A$,
\begin{multline*}
\sum_{n \in \N} J\left(\frac{a}{n + 1}\right)/n  = \sum_{n = a/A}^{a A} J\left(\frac{a}{n + 1}\right)/n  + O\left( \sum_{n < a/A} 1/\sqrt{a n} + \sum_{n > a A} a/n^2\right) \\
= \sum_{n = a/A}^{a A} J\left(\frac{a}{n + 1}\right)/n  + O(1/\sqrt{A}). 
\end{multline*}
Now, the sum
\[ \sum_{n = a/A}^{a A} J\left(\frac{a}{n + 1}\right)/n = \frac{1}{a} \sum_{n = a/A}^{a A} J\left(\frac{a}{n + 1}\right) \left(\frac{a}{n}\right)  \] is a Darboux sum of the function $f(x) = J(1/x) /x$ on the region $[A, 1/A]$ with distance $1/a$, so it converges to the integral 
\[\int_{1/A}^A J(1/x)/x dx =  \int_{1/A}^A J(x)/x dx. \]  Taking $A$ to $\infty$, this converges to $1/4$. 

To conclude, 

\[-I + II - III = \frac{\Delta}{4} - \frac{\Delta}{4} + \frac{\Delta}{4} + o(1) = \frac{\Delta}{4} + o(1),\] and

\[(\ref*{JJ})=o(1) + \sum_{s \in \N} \frac{1}{s^2} \frac{\Delta}{4} \frac{\alpha }{2 \pi}   (c^2 - 1) = \frac{c^2 - 1}{4} + o(1). \] This cancels out with the denominator to exactly $1/2$.

\end{proof}

\section{Universal Function Sign Changes}

In this section we analyze the sign changes of the function 
\[ \mathcal{M}(T) = \sum_{d } \frac{Q(d) \sqrt{d}}{s^{3/2}} \sum_{s \in \N}\cos \left(\frac{4 \pi s T}{d} - \frac{3 \pi}{4} \right) ,\]
which is the universal function from Theorem \ref*{tm3}, scaled down by the order of growth $y^{1/4}$ and with a change of coordinates $y = T^2$. In particular we provide a computational proof of Corollary \ref*{cor1}.

We will analyze this sum by truncating it on $d$. For a set $\mathcal{D}$ of square-free integers, define
\[\mathcal{M}_\mathcal{D}(T) = \sum_{d \in \mathcal{D} } \frac{Q(d) \sqrt{d}}{s^{3/2}} \sum_{s \in \N}\cos \left(\frac{4 \pi s T}{d} - \frac{3 \pi}{4} \right).\]
Observe that for a finite $\mathcal{D}$, this function becomes periodic, with a period
\[\mathcal{P}(\mathcal{D}) = \frac{1}{2}\mathrm{lcm}\{d: d \in \mathcal{D} \} \] and the error term from truncating $d$ is trivially bounded by 
\[\mathcal{E}(\mathcal{D}):= \sum_{d \not \in \mathcal{D} } Q(d) \sqrt{d} M_{max},\] where 
\[M_{max}:= \max_{x \in \R} \abs{ \sum_{s\in \N} \cos(x s - 3 \pi/4)s^{-3/2} }.\]
The function 

\[f(x) :=\sum_{s\in \N} \cos(4 \pi x s - 3 \pi/4)s^{-3/2}  = \frac{1}{2} (-1)^{3/4} \left(\mathrm{PolyLog}(3/2, e^{-4 \pi i x}) + i \mathrm{PolyLog}(3/2, e^{4 \pi i x})\right)\]
has period $1/2$ and is maximized at $0$ in absolute value: 
\begin{figure}[H]
  \centering
  % include first image
\includegraphics[width = \textwidth/2]{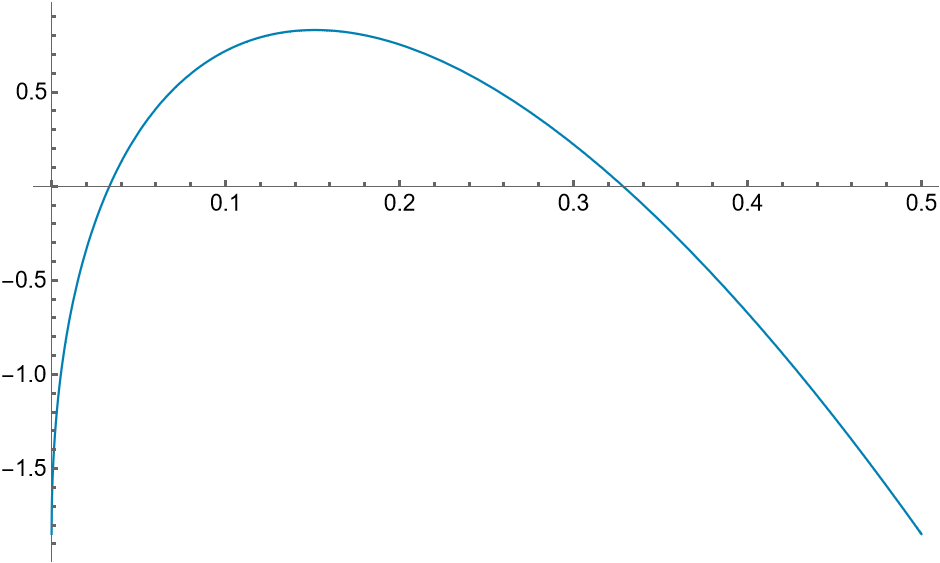}
  \caption{Graph of $f(x)$ for $x \in [0, 1/2]$}
\label{gr}
\end{figure}
hence
\[M_{max}= \frac{\sqrt{2}}{2} \zeta(3/2) \approx 1.84723.\]
We let 
\[\mathcal{D}:= \{1, 2, 3, 5, 6, 7, 10, 11, 13, 14, 15, 21, 22, 26, 30, 33, 35, 39, 42, 55, 65, 66, 70, 77, 78\},\] for which 
\[\mathcal{P}(\mathcal{D}) = \prod_{2 < p \leq 13} = 15015\] and 
\[\mathcal{E}(\mathcal{D}) \leq \left( 3.0907\footnotemark  -  \sum_{d \in \mathcal{D}} Q(d) \sqrt{d} \right) \mathcal{M} \leq 0.6306.\]
\footnotetext{the Euler product $\prod_{p < 10^8} (1 + Q(p) \sqrt{p})$ is $3.09064$ up to the $10^8$th prime. Note that for $p > 10^8$,  $Q(p) \sqrt{p} = 1/p^{3/2}\left(1 + \frac{2 p^2 + p - 1}{p^4 - 2 p^2 - p + 1} \right) \leq (1 + 10^{-16}) p^{-3/2}$; hence the multiplicative error term of this is at most $e^{ \sum_{p > 10^8} p^{-3/2}} \leq e^{2 \cdot 10^{-4}/\log 10^8} \leq 1 + 0.000011$ so this is within $0.00004$ of the limit.}
We then show that:

\begin{itemize}
\item for all $k \in \{1, \ldots, 15015\},$  $M_\mathcal{D}(k) < - \mathcal{E}(\mathcal{D})$;
\item for all $k \in \{1, \ldots, 15015\},$ $M_\mathcal{D}(k + 1/2) < - \mathcal{E}(\mathcal{D})$;
\item for all $k \in \{1, \ldots, 15015\},$ $M_\mathcal{D}(k + 0.162) >  \mathcal{E}(\mathcal{D})$,
\end{itemize}
where we use the upper numerical bound on $\mathcal{E}$.  
\newpage 

\begin{tcolorbox}[boxrule=0pt,sharp corners,parbox=false,breakable,colback=blue!10]
\begin{algorithmic}
\State $S \gets 0$
\State List $\gets  \mathcal{D}$
\For{$1 \leq k \leq 15015$} 
    \State $S \gets S + \mathrm{Boole}\left(\sum_{t = 1}^{\mathrm{Length}(\mathrm{List})} Q(\mathrm{List}[t]) \sqrt{\mathrm{List}[t]} f\left(\frac{k + 0.162}{\mathrm{List}[t]}\right) > 0.6306 \right)$
\EndFor \\
\Return $S$
\end{algorithmic}
\end{tcolorbox}
\begin{center}
{Pseudocode for the last example (executed with Mathematica)}
\end{center}
In all three cases, the conditions above are satisfied for all $1 \leq k \leq 15015$. From periodicity, it follows that the statement is true for all $k$. Since $\mathcal{M}_\mathcal{D}$ is within $\mathcal{E}(\mathcal{D})$ of $\mathcal{M}$ at every point, this proves positivity and negativity in these regions. 

From the graph of $M$, it appears as if there are two sign changes on every interval, with the function becoming positive again on $[k + 1/2, 1 + 1]$. However, while this is provably satisfied by the majority (at least $2/3$ of $k$'s between $1$ and $15015$, using the same method as above), this is not the case.

\begin{figure}[H]
  \begin{subfigure}[b]{3.2 in}
    \centering
    \includegraphics[width=3in]{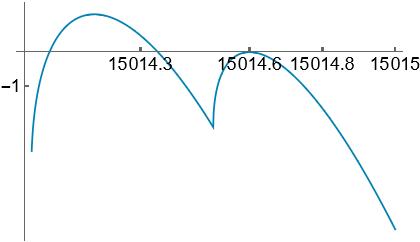}
  \end{subfigure}
  \hspace{0.11in}
  \begin{subfigure}[b]{3in}
    \centering
\vspace{-1in}
    \includegraphics[width=3.3 in]{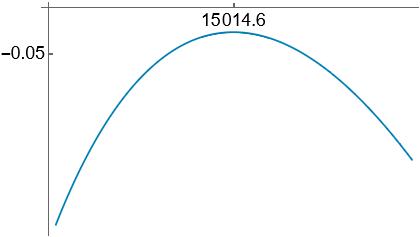}
  \end{subfigure}
\caption{Universal function $\mathcal{M}$ around $[15014, 15015]$ and at the second peak near $15014.6$.}
\end{figure}
\vspace{-0.1in} 

 For  $\mathcal{D'} = \{1, \ldots, 5000\}$, on the interval $[15014, 15014]$, the function $M_{\mathcal{D'}}$ does not appear to have a second sign change, as it attains a maximal value of $\approx - 0.27$ while the error term from this truncation cannot exceed $0.83  \left( 3.0907  -  \sum_{d >5000} Q(d) \sqrt{d} \right) \approx 0.022$, where $0.83 >  \max_{(0, 1/2)} f(x)$.

\end{proof}

\section{Acknowledgments}

I am very grateful to Andrew Sutherland for introducing this question and to Peter Sarnak for suggesting it as a project to me, as well as for many fruitful discussions. I would like to thank Jonathan Bober, Andrew Booker, Andrew Granville, Yang-Hui He, Min Lee, Michael Lipnowski, Mayank Pandey, Michael Rubinstein, and Will Sawin for enlightening conversations in connection to this problem, and Steven Wang for the many helpful comments on the first version of this paper.

The research was supported by NSF GRFP and the Hertz Fellowship.

\newpage

\bibliographystyle{plain}

\bibliography{murmurations}

\end{document}